\documentclass{louloupart1}
\usepackage{mathtools}
\usepackage{graphicx}
\usepackage{bm}
\usepackage{tikz}
\usetikzlibrary{shapes.geometric,calc}
\usetikzlibrary{decorations.markings,arrows,arrows.meta,positioning,cd}
\tikzset{->-/.style={decoration={
  markings,
  mark=at position #1 with {\arrow{Computer Modern Rightarrow[length=5pt,width=5pt]}}},postaction={decorate}}}
\tikzset{->-rev/.style={decoration={
  markings,
  mark=at position #1 with {\arrow{Computer Modern Rightarrow[length=5pt,width=5pt,reversed]}}},postaction={decorate}}}

\title[Endomorphisms of Artin groups of type $B_n$]{Endomorphisms of Artin groups of type $B_n$}
\author[L Paris]{Luis Paris}
\givenname{Luis}
\surname{Paris}
\address{Luis Paris, IMB, UMR 5584, CNRS, Université Bourgogne Europe, 21000 Dijon, France}
\email{lparis@u-bourgogne.fr}
\author[I Soroko]{Ignat Soroko}
\givenname{Ignat}
\surname{Soroko}
\address{Ignat Soroko, Department of Mathematics, University of North Texas, Denton, TX 76203-5017, USA}
\email{ignat.soroko@unt.edu, ignat.soroko@gmail.com}
\newtheorem{thm}{Theorem}[section]
\newtheorem{lem}[thm]{Lemma}
\newtheorem{prop}[thm]{Proposition}
\newtheorem{corl}[thm]{Corollary}
\newtheorem{conj}[thm]{Conjecture}

\newtheorem{problem}[thm]{Problem}
		
\newtheoremstyle{special}{13.2pt plus6.6pt minus6.6pt}{6.6pt plus3.3pt minus3.3pt}%
{\sl}{}{\bf}{}{0.75em}{\thmname{#1}\thmnumber{ #2$^{\bm*}\!$}\thmnote{\rm\stdspace(#3)}} 		
\theoremstyle{special}
\newtheorem{probstar}[thm]{Problem}

\theoremstyle{definition}

\newtheorem{rem}{Remark}

\newtheorem*{acknow}{Acknowledgments}

\numberwithin{equation}{section}

\makeatletter
\renewcommand{\thefigure}{\ifnum \c@section>\z@ \thesection.\fi
 \@arabic\c@figure}
\@addtoreset{figure}{section}
\makeatother

\begin{document}

\def\N{\mathbb N} \def\conj{{\rm conj}} \def\Aut{{\rm Aut}}
\def\Inn{{\rm Inn}} \def\Out{{\rm Out}} \def\Z{\mathbb Z}
\def\id{{\rm id}} \def\supp{{\rm supp}} %\def\Im{{\rm Im}}
\renewcommand{\Im}{\operatorname{Im}} 
\def\Ker{{\rm Ker}} \def\PP{\mathcal P} \def\Homeo{{\rm Homeo}}
\def\SHomeo{{\rm SHomeo}} \def\LHomeo{{\rm LHomeo}}
\def\MM{\mathcal M} \def\CC{\mathcal C} \def\AA{\mathcal A}
\def\S{\mathbb S} \def\FF{\mathcal F} \def\SS{\mathcal S}
\def\LL{\mathcal L} \def\D{\mathbb D} \def\Fix{{\rm Fix}}

\newcommand{\card}{\operatorname{Card}}
\newcommand{\Sym}{\operatorname{Sym}}
\newcommand{\Cent}{\operatorname{Cent}}
\newcommand{\ov}{\overline}
\mathchardef\mhyphen="2D
\renewcommand{\ge}{\geqslant}
\renewcommand{\le}{\leqslant}

\begin{abstract}
We determine a classification of the endomorphisms of the Artin groups of spherical type $B_n$ for $n\ge 5$, and of their quotients by the center. 

\smallskip\noindent
{\bf AMS Subject Classification\ \ } 
Primary: 20F36, secondary: 57K20.

\smallskip\noindent
{\bf Keywords\ \ } 
Artin groups of type $B_n$, endomorphisms, automorphisms.

\end{abstract}

\maketitle

%%%%%%%%%

\section{Introduction}\label{sec1}

Let $S$ be a finite set.
A \emph{Coxeter matrix} over $S$ is a square matrix $M=(m_{s,t})_{s,t\in S}$ indexed by the elements of $S$, with coefficients in $\N \cup \{\infty\}$, such that $m_{s,s}=1$ for all $s \in S$, and $m_{s,t} = m_{t,s} \ge 2$ for all $s,t\in S$, $s\neq t$.
Such a matrix is usually represented by a labeled graph, $\Gamma$, called a \emph{Coxeter graph}, defined by the following data.
The set of vertices of $\Gamma$ is $S$.
Two vertices $s,t\in S$ are connected by an edge if $m_{s,t}\ge 3$, and this edge is labeled with $m_{s,t}$ if $m_{s,t} \ge 4$.

If $a,b$ are two letters and $m$ is an integer $\ge 2$, then we denote by $\Pi(a,b,m)$ the alternating word $aba\!\ldots$ of length $m$.
In other words, $\Pi(a,b,m) = (ab)^{\frac{m}{2}}$ if $m$ is even, and $\Pi(a,b,m) = (ab)^{\frac{m-1}{2}} a$ if $m$ is odd.
Let $\Gamma$ be a Coxeter graph and let $M=(m_{s,t})_{s,t\in S}$ be its Coxeter matrix.
With $\Gamma$ we associate a group, $A[\Gamma]$, called the \emph{Artin group} of $\Gamma$, defined by the presentation
\[
A[\Gamma] = \langle S \mid \Pi(s,t,m_{s,t}) = \Pi(t,s,m_{s,t}) \text{ for } s, t \in S\,,\ s \neq t\,,\ m_{s,t}\neq \infty \rangle\,.
\]
The \emph{Coxeter group} of $\Gamma$, denoted by $W[\Gamma]$, is the quotient of $A[\Gamma]$ by the relations $s^2 = 1$, $s \in S$. We say that $\Gamma$ is \emph{of spherical type} if $W[\Gamma]$ is finite.

Despite the popularity of Artin groups, little is known about their automorphisms and even less about their endomorphisms.
The most studied cases are the braid groups, corresponding to the Coxeter graphs $A_n$ ($n\ge 1$), and the right-angled Artin groups, for which $m_{s,t}\in\{2,\infty\}$ for all $s,t\in S$.
The automorphism group of $A[A_n]$ was determined in~\cite{DyeGro1} and the set of its endomorphisms in~\cite{Caste1} for $n\ge 5$, in~\cite{ChKoMa1} for $n \ge 4$ and in~\cite{Orevk1} for $n\ge2$.
On the other hand, there are many papers dealing with automorphism groups of right-angled Artin groups, but little is known about their endomorphisms.

Apart from these two classes, the Artin groups for which the automorphism group is determined are the $2$-generator Artin groups~\cite{GiHoMeRa1}, some $2$-dimensional Artin groups~\cite{Crisp1,AnCho1}, the large-type free-of-infinity Artin groups~\cite{Vasko1}, the Artin groups of type $B_n$, $\tilde A_n$ and $\tilde C_n$~\cite{ChaCri1}, the Artin group of type $D_4$~\cite{Sorok1}, and the Artin groups of type $D_n$ for $n\ge 6$~\cite{CasPar1}.
On the other hand, apart from Artin groups of type $A_n$, 
the set of endomorphisms is determined only for Artin groups of type $D_n$ for $n\ge6$~\cite{CasPar1}, and, very recently, for the Artin groups of type $\tilde A_{n}$ for $n\ge 4$~\cite{ParSor1}.

The goal of the present paper is to determine a classification of the endomorphisms of the Artin groups of spherical type $A[B_n]$ for $n\ge 5$, and of its quotients by the center $A[B_n]/Z(A[B_n])$ (see Theorems~\ref{thm2_1} and \ref{thm2_8}), where $B_n$ is the Coxeter graph depicted in Figure~\ref{fig:CoxBn}. 
This paper is a continuation of the article~\cite{ParSor1} by the authors, in which, as a key ingredient, a classification of homomorphisms $A[\tilde A_{n-1}]\longrightarrow A[A_{n}]$ was obtained (see Proposition~4.1 therein). We apply this classification to the study of endomorphisms of $A[B_n]$ by considering a well-known sequence of inclusions: $A[\tilde A_{n-1}]\lhook\joinrel\longrightarrow A[B_{n}]\lhook\joinrel\longrightarrow A[A_{n}]$ (see Section~\ref{sec2}). Thus an arbitrary endomorphism $\varphi\colon A[B_n]\longrightarrow A[B_n]$ gives rise to a composition homomorphism
\[
A[\tilde A_{n-1}]\lhook\joinrel\longrightarrow A[B_{n}]\xrightarrow{\,\,\,\,\varphi\,\,\,\,} A[B_n]\lhook\joinrel\longrightarrow A[A_{n}],
\]
which can be analyzed using results from~\cite{ParSor1}. The techniques we use are mostly algebraic, with an occasional involvement of combinatorics of essential reduction systems of curves for the mapping classes of homeomorphisms of the punctured disk.

\begin{figure}[ht!]
\begin{center}
\begin{tikzpicture}[scale=0.7]
\begin{scope}[scale=1.33] 
\fill (0,0) circle (2.3pt) node [below=2pt,blue] {\scriptsize$1$};
\fill (1,0) circle (2.3pt) node [below=2pt,blue] {\scriptsize$2$};
\fill (2,0) circle (2.3pt) node [below=2pt,blue] {\scriptsize$3$};
\fill (4,0) circle (2.3pt) node [below=1pt,blue] {\scriptsize $n{-}1$};
\fill (5,0) circle (2.3pt) node [below=2pt,blue] {\scriptsize $n$};
\draw [very thick] (0,0)--(1,0)--(2,0)--(2.5,0);
\draw [very thick,dashed] (2.5,0)--(3.5,0);
\draw [very thick] (3.5,0)--(4,0)--(5,0);
\draw (4.5,0.3) node {\small $4$};
\draw (-1,0) node {$B_n$:};	
\end{scope}
\end{tikzpicture}
\end{center}
\caption{The Coxeter graph of type $B_n$ ($n\ge2$)\label{fig:CoxBn}}
\end{figure}
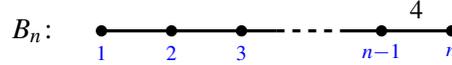

The paper is organized as follows.
In Section~\ref{sec2} we give definitions and precise statements of our results.
Section~\ref{sec3} contains preliminaries, Section~\ref{sec4} contains the proofs related to the endomorphisms of $A[B_n]$ and Section~\ref{sec5} contains proofs related to endomorphisms of $A[B_n]/Z(A[B_n])$. Section~\ref{sec_rem} contains some additional remarks, and in Section~\ref{sec_oq} we summarize the questions related to classification of automorphisms and endomorphisms of Artin groups of spherical and affine types which, to the best of our knowledge, remain open.

\begin{acknow}
This work was started at the AIM workshop ``Geometry and topology of Artin groups'' organized by Ruth Charney, Kasia Jankiewicz, and Kevin Schreve in September 2023, and continued at the SLMath workshop ``Hot Topics: Artin Groups and Arrangements -- Topology, Geometry, and Combinatorics'' organized by Christin Bibby, Ruth Charney, 
Giovanni Paolini, and Mario Salvetti in March 2024. We thank the organizers, the AIM and SLMath. We also thank the referee for carefully reading the manuscript and suggesting several valuable improvements.
The second author acknowledges support from the AMS--Simons travel grant.
\end{acknow}

%%%%%%%%%%

\section{Definitions and statements}\label{sec2}

\subsection{Key definitions}
As was mentioned above, two other Coxeter graphs play an important role in our study: the Coxeter graphs $A_n$ and $\tilde A_{n-1}$, depicted in Figure~\ref{fig2_1}. Our key ingredient is the following sequence of inclusions:
\[
A[\tilde A_{n-1}]\xhookrightarrow{\quad\iota_{\tilde A}\quad} A[B_{n}]\xhookrightarrow{\quad\iota_{B}\quad}  A[A_{n}]\,,
\]
which we are going to define. We denote the standard generators of $A[B_n]$ as $r_1,\dots,r_n$, the standard generators of $A[A_n]$ as $s_1,\dots,s_n$, and the standard generators of $A[\tilde A_{n-1}]$ as $t_0,t_1,\dots,t_{n-1}$.

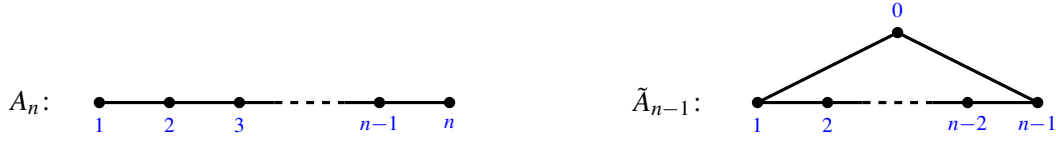
\begin{figure}
\begin{center}
\begin{tikzpicture}[scale=0.7]
% A_n
\begin{scope}[xshift=-6cm, yshift=0cm, scale=1.33] 
\fill (0,0) circle (2.3pt) node [below=2pt,blue] {\scriptsize$1$};
\fill (1,0) circle (2.3pt) node [below=2pt,blue] {\scriptsize$2$};
\fill (2,0) circle (2.3pt) node [below=2pt,blue] {\scriptsize$3$};
\fill (4,0) circle (2.3pt) node [below=1pt,blue] {\scriptsize$n{-}1$};
\fill (5,0) circle (2.3pt) node [below=2pt,blue] {\scriptsize$n$};
\draw [very thick] (0,0)--(1,0)--(2,0)--(2.5,0);
\draw [very thick,dashed] (2.5,0)--(3.5,0);
\draw [very thick] (3.5,0)--(4,0)--(5,0);
\draw (-1,0) node {$A_n$:};	
\end{scope}
% ~A_n
\begin{scope}[xshift=6.5cm, scale=1.33] 
\fill (0,0) circle (2.3pt) node [below=2pt,blue] {\scriptsize$1$}; 
\fill (1,0) circle (2.3pt) node [below=2pt,blue] {\scriptsize$2$}; 
\fill (3,0) circle (2.3pt) node [below=1pt,blue] {\scriptsize$n{-}2$}; 
\fill (4,0) circle (2.3pt) node [below=1pt,blue] {\scriptsize$n{-}1$}; 
\fill (2,1) circle (2.3pt) node [above=2pt,blue] {\scriptsize$0$}; 
\fill [white] (2,2) circle (2.4pt);
\draw [very thick](0,0)--(1,0)--(1.5,0);
\draw [very thick,dashed] (1.5,0)--(2.5,0);
\draw [very thick](2.5,0)--(3,0)--(4,0);
\draw [very thick](0,0)--(2,1)--(4,0);
\draw (-1.3,0) node {$\tilde A_{n-1}$:};	
\end{scope}
\end{tikzpicture}
\end{center}
\caption{Coxeter graphs $A_n$ ($n\ge1$), and $\tilde A_{n-1}$ ($n\ge3$)\label{fig2_1}}
\end{figure}

Define $\rho_B\in A[B_n]$ as
\[
\rho_B = r_1\, \ldots\, r_{n-1}\, r_{n}.
\]
A direct calculation shows that $\rho_B\, r_i\, \rho_B^{-1} = r_{i+1}$ for all $1\le i\le n-2$ and $\rho_B^2\,r_{n-1}\,\rho_B^{-2} = r_1$.
Thus if we set $r_0=\rho_B\, r_{n-1}\,\rho_B^{-1}$, then $\rho_B\,r_i\,\rho_B^{-1}=r_{i+1}$ for all $0\le i\le n-2$ and $\rho_B\,r_{n-1}\rho_B^{-1}=r_0$. (We note, however, that $r_{n}$ and $r_0$ are two different elements of $A[B_{n}]$, since their images in the abelianization of $A[B_{n}]$ are distinct.) It turns out (see~\cite{KenPei1,ChaCri1}) that the subgroup $\langle r_0,\dots,r_{n-1}\rangle$ of $A[B_n]$ is isomorphic to $A[\tilde A_{n-1}]$, and there is a decomposition of $A[B_n]$ into a semidirect product:
\[
A[B_n]\simeq A[\tilde A_{n-1}]\rtimes\Z=\langle r_0,\dots,r_{n-1}\rangle\rtimes\langle\rho_B\rangle,
\]
such that conjugation by $\rho_B$ induces the cyclic shift on generators $r_0,\dots,r_{n-1}$. We define $\iota_{\tilde A}$ as:
\[
\iota_{\tilde A}\colon A[\tilde A_{n-1}] \lhook\joinrel\longrightarrow A[B_{n}], \quad t_i\longmapsto r_i,\text{\quad for all\quad}
0\le i\le n-1.
\]
From now on we identify $A[\tilde A_{n-1}]$ with the image of $\iota_{\tilde A}$.
So, $A[\tilde A_{n-1}]$ is viewed as the subgroup of $A[B_{n}]$ generated by $r_0, r_1, \dots, r_{n-1}$, and we have $t_i = r_i$ for all $0 \le i \le n-1$ via this identification.

We define the second embedding $\iota_B\colon A[B_n]\lhook\joinrel\longrightarrow A[A_{n}]$ as:
\[
\iota_{B}\colon A[B_n] \lhook\joinrel\longrightarrow A[A_{n}], \quad r_i\longmapsto s_i,\text{\quad for all\quad}
1\le i\le n-1,\quad r_{n}\longmapsto (s_n)^2.
\]
The existence of an embedding $A[B_n]\lhook\joinrel\longrightarrow A[A_{n}]$ was observed independently by several authors (see~\cite{Bries1,Lambr1,Crisp2}). A simple combinatorial proof that $\iota_B$ is an embedding is given in~\cite[Proposition~1]{Manfr1}.

From now on we identify $A[B_{n}]$ with the image of $\iota_B$.
So, $A[B_{n}]$ is viewed as the subgroup of $A[A_{n}]$ generated by $s_1,\dots,s_{n-1},(s_{n})^2$, and we have $r_i=s_i$ for all $1\le i\le n-1$ and $r_{n}=(s_{n})^2$ via this identification.

The Coxeter graphs $B_n$ and $A_n$ are of spherical type, while the Coxeter graph $\tilde A_{n}$ is not of spherical type (it is of \emph{affine} type). We know that the center of $A[\tilde A_{n}]$ is trivial (see~\cite[Proposition 1.3]{ChaPei1}), whereas the centers of $A[A_n]$ and $A[B_n]$ are infinite cyclic. The center $Z(A[A_n])$ of $A[A_n]$ is generated by $\Delta^2$, where $\Delta$ 
is the so-called Garside element of $A[A_n]$ equal to:
\[
\Delta=(s_1s_2\ldots s_n)\,(s_1s_2\ldots s_{n-1})\,\ldots\, (s_1s_2)\,s_1\,,
\]
see~\cite[Theorem~1.24]{KasTur1}. We have $\Delta s_i \Delta^{-1} = s_{n+1-i}$ for all $1\le i\le n$. Similarly, the center $Z(A[B_n])$ of $A[B_n]$ is generated by the Garside element $\Delta_B$ 
of $A[B_n]$, which is equal to:
\[
\Delta_B=(r_1\ldots r_{n-1}r_n)^n=\rho_B^n\,,
\]
see~\cite[Satz~7.2]{BriSai1} and \cite[Chapter\,VI,\,Section~4,\,n$^\circ5$,\,(III)]{Bourb1}. Moreover, under the identification via the embedding $\iota_B$, we have $\Delta_B=\Delta^2$, so that $A[B_n]$ and $A[A_n]$ share the same center, see~\cite[Lemma~4.2]{ParSor1}.

Let, as before, $\Gamma$ be a Coxeter graph with the vertex set $S$. For $X\subset S$ we denote by $\Gamma_X$ the full subgraph of $\Gamma$ spanned by $X$, by $A_X[\Gamma]$ the subgroup of $A[\Gamma]$ generated by $X$, and by $W_X[\Gamma]$ the subgroup of $W[\Gamma]$ generated by $X$. We know by~\cite{Lek1} that $A_X[\Gamma]$ is naturally isomorphic to $A[\Gamma_X]$, and we know by~\cite[Chapter~4,\,Section~1.8,\,Theorem~2(i)]{Bourb1} that $W_X[\Gamma]$ is naturally isomorphic to $ W[\Gamma_X]$.

Denote $Y=\{t_1,t_2,\dots,t_{n-1}\}\subset A[\tilde A_{n-1}]$. Notice that the full subgraph of $\tilde A_{n-1}$ spanned by $Y$ is isomorphic to $A_{n-1}$. We denote by $\Delta_Y=\Delta_Y[\tilde A_{n-1}]$ the Garside element of $A_Y[\tilde A_{n-1}]$. In particular,
\[
\Delta_Y=(t_1t_2\ldots t_{n-1})\,(t_1t_2\ldots t_{n-2})\,\ldots\, (t_1t_2)\,t_1\,,
\]
and $\Delta_Y^2$ generates the center of $A_Y[\tilde A_{n-1}]$.
Let
\[
\rho_0=t_1\,t_2\,\ldots\, t_{n-1},\qquad \rho_1=t_1^{-1}\,t_2^{-1}\,\ldots\, t_{n-1}^{-1}.
\]
Clearly $\rho_0,\rho_1\in A_Y[\tilde A_{n-1}]$. A direct calculation shows that $\rho_0\, t_i\,\rho_0^{-1}=\rho_1\,t_i\,\rho_1^{-1}=t_{i+1}$ for all $1\le i\le n-2$ and $\rho_0^2\, t_{n-1}\,\rho_0^{-2}=\rho_1^2\,t_{n-1}\,\rho_1^{-2}=t_1$. So if we define
\[
v_0=\rho_0\, t_{n-1}\,\rho_0^{-1},\qquad v_1=\rho_1\,t_{n-1}\,\rho_1^{-1},
\]
then $\rho_k\, t_{n-1}\,\rho_k^{-1}=v_k$ and $\rho_k\, v_k\,\rho_k^{-1}=t_1$, for $k=0,1$. Figure~\ref{fig_v01} depicts the standard generators $t_1,\dots,t_{n-1}$ and elements $v_0$ and $v_1$ interpreted as braids on $n+1$ strands (viewed as elements of $A[A_n]$ via the embedding $\iota_B\circ\iota_{\tilde A}$). 

\begin{figure}[ht!]
\begin{center}
\includegraphics[width=1.6cm]{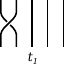}\hskip 1truecm
\includegraphics[width=1.6cm]{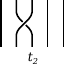}\hskip 1truecm
\includegraphics[width=1.6cm]{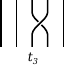}\hskip 1truecm
\includegraphics[width=1.6cm]{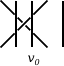}\hskip 1truecm
\includegraphics[width=1.6cm]{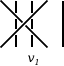}
\caption{Braid pictures of the standard generators $t_1,\dots,t_{n-1}$ and elements $v_0$ and $v_1$, depicted for $n=4$.\label{fig_v01}}
\end{center}
\end{figure}

Also, we denote by $\delta\in A[B_n]$ the element (recall that we identify $t_i$ with $r_i$ for $0\le i\le n-1$):
\[
\delta=(\rho_1)^{-1}\rho_0=r_{n-1}r_{n-2}\dots r_2\,(r_1^2)\,r_2\dots r_{n-2}r_{n-1},
\]
which will be used in the description of endomorphisms and automorphisms of $A[B_n]$ written in the standard generators in Corollaries~\ref{cor2_3} and \ref{cor2_6}. We note also that the word representing element $\delta$ appears in some other contexts in the literature, see Remark~\ref{rem:delta} in Section~\ref{sec_rem}.

If $G$ is a group and $g\in G$, then we denote by $\conj_g\colon G\to G$, $h \mapsto g h g^{-1}$, the conjugation by $g$.
We have a homomorphism $\conj\colon G\to\Aut(G)$, $g\mapsto\conj_g$, whose image is the group $\Inn(G)$ of \emph{inner automorphisms}, and whose kernel is the center of $G$.
The group $\Inn(G)$ is a normal subgroup of $\Aut(G)$, and the quotient $\Out(G)=\Aut(G)/\Inn(G)$ is the \emph{outer automorphism group} of $G$.
Two homomorphisms $\varphi_1,\varphi_2\colon G\to H$ are said to be \emph{conjugate} if there exists $h\in H$ such that $\varphi_2=\conj_h\circ\varphi_1$. Our classification of endomorphisms of $A[B_n]$ and $A[B_n]/Z(A[B_n])$ will be made up to conjugation.

\subsection{Results on endomorphisms of \texorpdfstring{$A[B_{n}]$}{A[Bn]}}
Now we are ready to state our results on endomorphisms of $A[B_n]$. Their proofs will be given in Section~\ref{sec4}.

\begin{thm}\label{thm2_1}
Let $n\ge5$ and $\varphi\colon A[B_{n}]\longrightarrow A[B_{n}]$ be an endomorphism. Then, up to conjugation, $\varphi$ belongs to one of the following three types, written in generators $t_0,\dots,t_{n-1},\rho_B$ as follows:
\begin{itemize}
\item[(1)] There exist $g,h\in A[B_{n}]$ such that $gh=hg$ and for all $0\le i\le n-1$:
\begin{align*}
\varphi(t_i)&=g,\\
\varphi(\rho_B)&=h.
\end{align*}
\item[(2)] There exist $\varepsilon\in\{\pm1\}$ and $p,q\in\Z$ such that for all $0\le i\le n-1$: 
\[
\text{(2a):\qquad}
\begin{aligned}
\varphi(t_i)&=t_i^\varepsilon\,\Delta_B^{p}\,,\\
\varphi(\rho_B)&=\rho_B\,\Delta_B^{q}\,,
\end{aligned}
\qquad\qquad \text{or} \qquad\qquad \text{(2b):\qquad}
\begin{aligned}
\varphi(t_i)&=t_{n-i}^\varepsilon\,\Delta_B^{p}\,,\\
\varphi(\rho_B)&=\rho_B^{-1}\,\Delta_B^{q}.
\end{aligned}
\]
(Note that the indices are taken modulo $n$ so that $t_n=t_0$.)
\item[(3)] There exist $\varepsilon\in\{\pm1\}$, $k\in\{0,1\}$ and $p,q,r,s\in\Z$ such that for all $1\le i\le n-1$:
\[
\begin{aligned}
\varphi(t_i)&=t_i^\varepsilon\,\Delta_Y^{2p}\Delta_B^q\,,\\
\varphi(t_0)&=v_k^\varepsilon\,\Delta_Y^{2p}\Delta_B^q\,,\\
\varphi(\rho_B)&=\rho_k\,\Delta_Y^{2r}\Delta_B^s\,.
\end{aligned}
\]
\end{itemize}
\end{thm}

Except for endomorphisms of type (1), where the pair of commuting elements $(g,h)$ is determined only up to simultaneous conjugation in $A[B_n]$, there is no redundancy in Theorem~\ref{thm2_1}, as the following proposition shows.
\begin{prop}\label{prop2_1}
Let $n\ge5$ and $\varphi$ and $\psi$ be two endomorphisms of $A[B_n]$, such that for some $x\in A[B_n]$ we have $\varphi=\conj_x\circ \psi$. Then $\varphi$ and $\psi$ belong to the same type (1), (2a), (2b), or (3) of Theorem~\ref{thm2_1}, and parameters $\varepsilon\in\{\pm1\}$, $k\in\{0,1\}$, $p,q,r,s\in\Z$ are the same for $\varphi$ and $\psi$.
\end{prop}

We also give a description of endomorphisms of $A[B_n]$ in the standard generators.

\begin{corl}\label{cor2_3}
In the standard generators $r_1,\dots,r_{n}$ of $A[B_{n}]$ any endomorphism $\varphi\colon A[B_{n}]\longrightarrow A[B_{n}]$, viewed up to conjugation in $A[B_{n}]$, can be written as follows:
\begin{itemize}
\item[(I)] There exist $g,h\in A[B_{n}]$ such that $gh=hg$ and for all $1\le i\le n-1$:
\begin{align*}
\varphi(r_i)&=g,\\ 
\varphi(r_{n})&=h.
\end{align*}
\item[(II)] There exist $\varepsilon\in\{\pm1\}$ and $p,q\in\Z$ such that for all $1\le i\le n-1$:
\[
\text{(IIa):\qquad}
\begin{aligned}
\varphi(r_i)&=r_i^\varepsilon\,\Delta_B^{p}\\
\varphi(r_{n})&=r_{n}^\varepsilon\,\Delta_B^{q}
\end{aligned}
\qquad\qquad \text{or} \qquad\qquad \text{(IIb):\qquad}
\begin{aligned}
\varphi(r_i)&=r_i^\varepsilon\,\Delta_B^{p}\\
\varphi(r_{n})&=\delta^{-\varepsilon}\,r_{n}^{-\varepsilon}\,\Delta_B^{q}.
\end{aligned}
\]
\item[(III)] There exist $\varepsilon\in\{\pm1\}$ and $p,q,r,s\in\Z$ such that for all $1\le i\le n-1$:
\[
\text{(IIIa):\qquad}
\begin{aligned}
\varphi(r_i)&=r_i^\varepsilon\,\Delta_Y^{2p}\Delta_B^q\\
\varphi(r_{n})&=\Delta_Y^{2r}\Delta_B^s
\end{aligned}
\qquad\qquad \text{or} \qquad\qquad \text{(IIIb):\qquad}
\begin{aligned}
\varphi(r_i)&=r_i^\varepsilon\,\Delta_Y^{2p}\Delta_B^q\\
\varphi(r_{n})&=\delta^{-\varepsilon}\,\Delta_Y^{2r}\Delta_B^s.
\end{aligned}
\]
\end{itemize}
\end{corl}

\begin{rem}
It will follow from the proof of Theorem~\ref{thm2_1}, that case (IIa) of Corollary~\ref{cor2_3} combines case $\varepsilon=1$ of (2a) and case $\varepsilon=-1$ of (2b) of Theorem~\ref{thm2_1}, whereas case (IIb) combines case $\varepsilon=1$ of (2b) and $\varepsilon=-1$ of (2a). Similarly, case (IIIa) of Corollary~\ref{cor2_3} combines case $\varepsilon=1$, $k=0$ and case $\varepsilon=-1$, $k=1$ of (3), whereas case (IIIb) combines case $\varepsilon=-1$, $k=0$ and case $\varepsilon=1$, $k=1$ of (3) in Theorem~\ref{thm2_1}.
\end{rem}

\begin{corl}\label{cor2_4}
Let $n\ge5$ and $\varphi\colon A[B_{n}]\longrightarrow A[B_{n}]$ be an endomorphism. 
\begin{enumerate}
\item $\varphi$ is injective if and only if $\varphi$ is conjugate to an endomorphism of type (2a) or (2b) of Theorem~\ref{thm2_1}.
\item $\varphi$ is surjective if and only if $\varphi$ is conjugate to an endomorphism of type (2a) or (2b) of Theorem~\ref{thm2_1} with $q=0$.
\end{enumerate}
\end{corl}

As another corollary of Theorem~\ref{thm2_1}, we recover a description of the automorphism group and the outer automorphism group of $A[B_{n}]$ for $n\ge5$. This result was obtained for $n\ge3$ in~\cite{ChaCri1}, and, independently, in~\cite[Theorem~7]{BelMar2}, by different methods.
\begin{corl}[Charney--Crisp~\cite{ChaCri1}, Bell--Margalit~\cite{BelMar2}]\label{cor2_5}
Let $n\ge5$. Then 
\[
\Aut(A[B_{n}])\simeq(\Inn(A[B_{n}])\rtimes\Z)\rtimes(\Z/2\Z\times\Z/2\Z)\,,
\]
and
\[
\Out(A[B_{n}])\simeq(\Z\rtimes \Z/2\Z)\times \Z/2\Z\simeq D_\infty\times \Z/2\Z\,,
\]
where the $\Z$ factor is generated by the transvection automorphism $T$: 
\[
T\colon\quad
\begin{aligned}
&T(t_i)=t_i\cdot\Delta_B,\quad 0\le i\le n-1,\\
&T(\rho_B)=\rho_B\,,
\end{aligned}
\]
and the two copies of $\Z/2\Z$ are generated by commuting involutions $\tau$ and $\mu$: 
\[
\tau\colon\quad
\begin{aligned}
&\tau(t_i)=t_{n-i}^{-1}\,,\\
&\tau(\rho_B)=\rho_B^{-1}\,,
\end{aligned}
\qquad\qquad
\mu\colon\quad
\begin{aligned}
&\mu(t_i)=t_i^{-1}\,,\\
&\mu(\rho_B)=\rho_B\,,
\end{aligned}
\]
for all $0\le i\le n-1$, i.e.\
\[
\begin{aligned}
\Aut(A[B_{n}])&\simeq\bigl(\Inn(A[B_{n}])\rtimes\langle T\rangle\bigr)\rtimes\bigl(\langle \mu\rangle\times\langle\tau\rangle\bigr)\,,\\
\Out(A[B_{n}])&\simeq\langle T\rangle\rtimes \bigl(\langle \mu\rangle\times\langle\tau\rangle\bigr)\simeq\bigl(\langle T\rangle\rtimes \langle \mu\rangle\bigr)\times\langle\tau\rangle\,,
\end{aligned}
\]
with $D_\infty=\langle T\rangle\rtimes \langle \mu\rangle$ being isomorphic to the infinite dihedral group.
\end{corl}

We also get a description of automorphisms of $A[B_{n}]$ given in the standard generators in Corollary~\ref{cor2_6} below. The proof of it follows from Corollaries~\ref{cor2_3} and \ref{cor2_5} by direct computation using the relation $r_{n}=r_{n-1}^{-1}\dots r_1^{-1}\cdot\rho_B$.

\begin{corl}\label{cor2_6}
Let $n\ge5$ and $\varphi\colon A[B_{n}]\longrightarrow A[B_{n}]$ be an automorphism. Then $\varphi$ is conjugate to an automorphism written in the standard generators $r_1,\dots,r_{n}$ of $A[B_{n}]$ in one of the following two forms:
\[
\text{(IIa):\qquad}
\begin{aligned}
\varphi(r_i)&=r_i^\varepsilon\cdot\Delta_B^{p}\\
\varphi(r_{n})&=r_{n}^\varepsilon\cdot\Delta_B^{-p(n-1)}
\end{aligned}
\qquad\qquad \text{or} \qquad\qquad \text{(IIb):\qquad}
\begin{aligned}
\varphi(r_i)&=r_i^\varepsilon\cdot\Delta_B^{p}\\
\varphi(r_{n})&=\delta^{-\varepsilon}r_{n}^{-\varepsilon}\cdot\Delta_B^{-p(n-1)},
\end{aligned}
\]
for some $\varepsilon\in\{\pm1\}$ and $p\in\Z$. The splitting generators $T$, $\tau$ and $\mu$ of $\Out(A[B_{n}])$ have the form (for $1\le i\le n-1$):
\[
T\colon\,\,
\begin{aligned}
&T(r_i)=r_i\cdot \Delta_B\,,\\
&T(r_{n})=r_{n}\cdot \Delta_B^{-(n-1)}\,,
\end{aligned}
\qquad
\tau\colon\,\,
\begin{aligned}
&\tau(r_i)=r_{n-i}^{-1}\,,\\
&\tau(r_{n})=r_1\dots r_{n-1}\cdot r_{n}^{-1} \cdot r_{n-1}^{-1} \dots r_1^{-1}\,,
\end{aligned}
\qquad
\mu\colon\,\,
\begin{aligned}
&\mu(r_i)=r_i^{-1}\,,\\
&\mu(r_{n})=\delta\cdot r_{n}\,.
\end{aligned}
\]
\end{corl}

\begin{rem}
The automorphism $\tau$ looks rather exotic when written in the standard generators, and the reader may wonder how $\tau$ can be presented in one of the forms (IIa) or (IIb) of Corollary~\ref{cor2_6}, up to conjugacy. Observing that $\Delta_Y \,r_i\,\Delta_Y^{-1}=r_{n-i}$ for all $i=1,\dots,n-1$, and taking into account Lemma~\ref{lem4_2} below, we conclude that $\tau=\conj_{\Delta_Y}\circ \tau'$, where $\tau'$ is of type (IIa) with $p=0$ and $\varepsilon=-1$, acting on the standard generators as follows: $\tau'(r_i)=r_i^{-1}$, for all $i=1,\dots,n$.
\end{rem}

\subsection{Results on endomorphisms of \texorpdfstring{$A[B_n]/Z(A[B_n])$}{A[Bn]/Z(A[Bn])}}
Here we state our results on endomorphisms of $A[B_n]/Z(A[B_n])$. Their proofs will be given in Section~\ref{sec5}.

Denote for brevity $\ov{A[B_{n}]}=A[B_{n}]/Z(A[B_{n}])$ and let $\pi\colon A[B_{n}]\to \ov{A[B_{n}]}$ be the canonical projection. For each $1\le i\le n$, we set $\bar r_{i}=\pi(r_i)$, and for each $0\le i\le n-1$, we set $\bar t_i=\pi(t_i)$. Recall that $Z(A[B_{n}])$ is generated by $\rho_B^{n}=\Delta_B$, and hence
\[
\ov{A[B_{n}]}\simeq A[\tilde A_{n-1}]\rtimes \langle \bar\rho_B\rangle, 
\]
where $\bar\rho_B=\pi(\rho_B)$ has order $n$. 

Note that if $\varphi\colon A[B_{n}]\to A[B_{n}]$ has the property $\varphi(Z(A[B_{n}]))\subseteq Z(A[B_{n}])$, it induces a natural endomorphism $\ov{\varphi}\colon \ov{A[B_{n}]}\to \ov{A[B_{n}]}$. 
Let $\phi\colon \ov{A[B_{n}]}\to \ov{A[B_{n}]}$ be an arbitrary endomorphism. If there exists an endomorphism $\varphi\colon A[B_{n}]\to A[B_{n}]$ making the following diagram commutative

\begin{center}
\setlength\mathsurround{0pt}
\begin{tikzcd}
A[B_{n}] \arrow[r,"\varphi"] \arrow[d,"\pi"] & A[B_{n}]\arrow[d,"\pi"]\\
\ov{A[B_{n}]}\arrow[r,"\phi"] & \ov{A[B_{n}]}
\end{tikzcd}
\end{center}

we say that $\phi$ \emph{lifts} and that $\varphi$ is a \emph{lift} of $\phi$. Notice that any lift $\varphi$ of any endomorphism $\phi$ of $\ov{A[B_{n}]}$ must have property $\varphi(Z(A[B_{n}]))\subseteq Z(A[B_{n}])$.

\begin{prop}\label{prop_lifts}
Let $n\ge5$. Then every endomorphism of $\ov{A[B_{n}]}$ lifts.
\end{prop}

Recall the automorphisms $\tau$ and $\mu$ of $A[B_{n}]$ defined in Corollary~\ref{cor2_5}. Projecting them to $\ov{A[B_{n}]}$, we get automorphisms $\bar\tau$, $\bar\mu$ of $\ov{A[B_{n}]}$:
\[
\bar\tau\colon\quad
\begin{aligned}
&\bar\tau(\bar t_i)=\bar t_{n-i}^{-1}\,,\\
&\bar\tau(\bar\rho_B)=\bar\rho_B^{-1}\,,
\end{aligned}
\qquad\qquad
\bar\mu\colon\quad
\begin{aligned}
&\bar\mu(\bar t_i)=\bar t_i^{-1}\,,\\
&\bar\mu(\bar\rho_B)=\bar\rho_B\,,
\end{aligned}
\]
for all $0\le i\le n-1$.

In what follows, $Z_G(g)=\{x\in G\mid xg=gx\}$ denotes the centralizer of an element $g$ in a group~$G$.
\begin{thm}\label{thm2_8}
Let $n\ge5$ and $\phi\colon \ov{A[B_{n}]}\to \ov{A[B_{n}]}$ be an endomorphism. Then we have one of the following two possibilities up to conjugation:
\begin{enumerate}
\item there exists $\kappa\in\{0,\dots,n-1\}$ such that:
\begin{align*}
\phi(\bar \rho_B)&=\bar\rho_B^\kappa\,,\\
\phi(\bar t_i)&=\bar g\,\quad\text{for all}\quad 0\le i\le n-1.
\end{align*}
with $\bar g\in Z_{\ov{A[B_{n}]}}(\bar\rho_B^\kappa)$. Moreover, denoting $d=\gcd(n,\kappa)$, we have the following description of the centralizer of $\bar\rho_B^\kappa$ in $\ov{A[B_{n}]}$ (with the identification $\bar t_{n}=\bar t_0$):
\[
Z_{\ov{A[B_{n}]}}(\bar\rho_B^\kappa)=
\begin{cases}
\langle \bar\rho_B\rangle,\text{\quad if\, $d=1$,}\\
\langle \bar\rho_B,\,\bar t_{d}\bar t_{2d}\bar t_{3d}\dots\bar t_0\rangle,\text{\quad if\, $d\ne1$.}
\end{cases}
\]
In particular, if $\kappa=0$, then $d=n$ and $Z_{\ov{A[B_{n}]}}(\bar\rho_B^\kappa)=\langle \bar\rho_B,\,\bar t_0\rangle=\ov{A[B_{n}]}$.
\item There exists $\varepsilon\in\{\pm1\}$ such that for all $0\le i\le n-1$:
\[
\text{(2a):\qquad}
\begin{aligned}
\phi(\bar t_i)&=\bar t_i^\varepsilon\,,\\
\phi(\bar \rho_B)&=\bar \rho_B\,,
\end{aligned}
\qquad\qquad \text{or} \qquad\qquad \text{(2b):\qquad}
\begin{aligned}
\phi(\bar t_i)&=\bar t_{n-i}^\varepsilon\,,\\
\phi(\bar \rho_B)&=\bar \rho_B^{-1}\,,
\end{aligned}
\]
i.e.\ $\phi\in\langle\bar\tau,\bar\mu\rangle$.
\end{enumerate}
\end{thm}

Except for the element $\bar g$ in case (1), which can be chosen up to conjugation in $Z_{\ov{A[B_{n}]}}(\bar\rho_B^\kappa)$, there is no other redundancy in Theorem~\ref{thm2_8}, as the following proposition shows.

\begin{prop}\label{prop2_9}
Let $n\ge5$ and $\phi$ and $\psi$ be two endomorphisms of $\ov{A[B_n]}$, such that for some $x\in \ov{A[B_n]}$ we have $\phi=\conj_x\circ \psi$. Then $\phi$ and $\psi$ belong to the same type (1), (2a), or (2b) of Theorem~\ref{thm2_8}, and parameters $\varepsilon\in\{\pm1\}$ and $\kappa\in\{0,\dots,n-1\}$ are the same for $\phi$ and $\psi$.
\end{prop}

As the first consequence of Theorem~\ref{thm2_8} we recover a description of the automorphism group of $\ov{A[B_{n}]}$, which was obtained in~\cite{ChaCri1} for $n\ge3$ by different methods.
\begin{corl}[Charney--Crisp~\cite{ChaCri1}]\label{cor2_10}
Let $n\ge5$. Then
\[
\Aut(\ov{A[B_{n}]})=\Inn(\ov{A[B_{n}]})\rtimes \langle\bar\tau,\bar\mu\rangle\simeq \ov{A[B_{n}]}\rtimes(\Z/2\Z\times\Z/2\Z)\,,
\]
and $\Out(\ov{A[B_{n}]})\simeq \Z/2\Z\times\Z/2\Z$.
\end{corl}

Recall that a group $G$ is called \emph{Hopfian} if every surjective endomorphism $G\to G$ is injective, and it is called \emph{co-Hopfian} if every injective endomorphism $G\to G$ is surjective. As another straightforward consequence of Theorem~\ref{thm2_8} we obtain the following.
\begin{corl}
Let $n\ge5$. Then $\ov{A[B_{n}]}$ is both Hopfian and co-Hopfian.
\end{corl}
\begin{proof}
Indeed, endomorphisms of case (1) of Theorem~\ref{thm2_8} are neither injective nor surjective, and endomorphisms of case (2) are automorphisms.
\end{proof}
Notice that the fact that $\ov{A[B_{n}]}$ is co-Hopfian was proved for $n\ge3$ in~\cite{BelMar2} by different methods. Also, that $\ov{A[B_n]}$ is Hopfian is also known, but the proof of it is not very straightforward. It relies on the fact that $A[B_n]$ admits a faithful linear representation (as a subgroup of the braid group), and that the quotient by the center of a linear group is linear (see~\cite[Theorem~6.4]{Wehrf1}).

Finally, recall that a subgroup $H$ of a group $G$ is called \emph{characteristic} in $G$ if for every automorphism $\varphi\in\Aut(G)$ we have $\varphi(H)=H$, and $H$ is called \emph{fully characteristic} in $G$ if for every endomorphism $\varphi\colon G\to G$ we have $\varphi(H)\subseteq H$. We get the following remarkable fact.
\begin{corl}\label{cor2_12}
Let $n\ge5$. Then the subgroup $A[\tilde A_{n-1}]=\langle \bar t_0,\dots,\bar t_{n-1}\rangle$ of $\ov{A[B_{n}]}$ is characteristic, but not fully characteristic in $\ov{A[B_{n}]}$.
\end{corl}
The fact that $A[\tilde A_{n-1}]$ is a characteristic subgroup of $\ov{A[B_{n}]}$ for $n\ge3$ was mentioned in~\cite[p.~328]{ChaCri1}, where this fact was established by different methods. Notice that $A[\tilde A_{n-1}]$ is not characteristic in $A[B_n]$ because of the presence of transvection automorphisms $T^k$.

%%%%%%%%%%

\section{Preliminaries}\label{sec3}

In this section we state auxiliary results which we need for proving main theorems.

The first one is a description of homomorphisms from $A[\tilde A_{n-1}]\to A[A_n]$ from~\cite{ParSor1}. Recall that the standard generators for $A[\tilde A_{n-1}]$ are denoted $t_0,t_1,\dots,t_{n-1}$, and the standard generators of $A[A_n]$ are denoted $s_1,\dots,s_n$, with $\Delta$ denoting the Garside element of $A[A_n]$. We also define two elements $u_0$, $u_1\in A[\tilde A_{n-1}]$:
\[
u_0=t_0,\qquad u_1=\Delta_Y^{-1}\,t_0\,\Delta_Y,
\]
where, as before, $\Delta_Y$ denotes the Garside element of $A_Y[\tilde A_{n-1}]=\langle t_1,\dots,t_{n-1}\rangle\simeq A[A_{n-1}]$.
Recall also that $v_0=\rho_0\, t_{n-1}\,\rho_0^{-1}$ and $v_1=\rho_1\, t_{n-1}\,\rho_1^{-1}$, as defined in Section~\ref{sec2}. We say that a group homomorphism $\varphi\colon G\to H$ is \emph{cyclic}, if its image is a cyclic subgroup of $H$. 

\begin{prop}[\protect{Paris--Soroko~\cite[Proposition~4.1]{ParSor1}}]\label{prop3_1}
Let $n \ge 5$.
Let $\varphi\colon A[\tilde A_{n-1}] \to A[A_{n}]$ be a homomorphism.
Then we have one of the following three possibilities:
\begin{itemize}
\item[(1)]
$\varphi$ is cyclic.
\item[(2)]
There exist $g \in A[A_{n}]$, $k \in \{0,1\}$, $\varepsilon \in \{\pm 1\}$ and $q \in \Z$ such that $\varphi(t_i) = g\, s_i^{\varepsilon}\, \Delta^{2q}\, g^{-1}$ for all $1 \le i \le n-1$, and $\varphi(t_0) = g\, u_k^{\varepsilon}\, \Delta^{2q}\, g^{-1}$.
\item[(3)]
There exist $g \in A[A_{n}]$, $k \in \{0,1\}$, $\varepsilon \in \{\pm 1\}$ and $p, q \in \Z$ such that $\varphi(t_i) = g\, s_i^{\varepsilon}\, \Delta_Y^{2p}\, \Delta^{2q}\, g^{-1}$ for all $1 \le i\le n-1$, and $\varphi(t_0) = g\, v_k^{\varepsilon}\, \Delta_Y^{2p}\, \Delta^{2q}\, g^{-1}$.\pushQED{\qed}\qedhere\popQED
\end{itemize}
\end{prop}

We will also need an interpretation of Artin groups $A[\tilde A_{n-1}]$, $A[B_n]$ and $A[A_n]$ as subgroups of the mapping class group of the punctured disk. So, we give below the information on mapping class groups that we will need, and we refer to~\cite{FarMar1} for a complete exposition on the subject.

Let $\Sigma$ be a compact oriented surface with or without boundary, and let $\PP$ be a finite family of punctures  in the interior of $\Sigma$.
We denote by $\Homeo^+(\Sigma,\PP)$ the group of homeomorphisms of $\Sigma$ which preserve the orientation, which are the identity on a neighborhood of the boundary, and which leave set-wise invariant the set $\PP$.
The \emph{mapping class group} of the pair $(\Sigma,\PP)$, denoted by $\MM(\Sigma,\PP)$, is the group of isotopy classes of elements of $\Homeo^+(\Sigma,\PP)$, where isotopies are required to leave the set $\partial \Sigma \cup \PP$ point-wise invariant. 
If $\PP=\varnothing$, then we write $\Homeo^+(\Sigma)=\Homeo^+(\Sigma,\varnothing)$ and $\MM(\Sigma)=\MM(\Sigma,\varnothing)$.

A \emph{circle} of $(\Sigma,\PP)$ is the image of an embedding $a\colon \S^1 \hookrightarrow \Sigma \setminus (\partial\Sigma \cup \PP)$.
It is called \emph{generic} if it does not bound any disk containing $0$ or $1$ puncture, and if it is not parallel to any boundary component.
The isotopy class of a circle $a$ is denoted by $[a]$. We emphasize that isotopies are considered in $\Sigma \setminus (\partial\Sigma \cup \PP)$, i.e.\ circles are not allowed to pass through points of $\partial\Sigma \cup\PP$ under isotopies. 
We denote by $\CC(\Sigma,\PP)$ the set of isotopy classes of generic circles.
The \emph{intersection index} of two classes $[a],[b]\in\CC(\Sigma,\PP)$ is $i([a],[b])=\min\{|a'\cap b'|\mid a'\in[a]\text{ and }b'\in[b]\}$.
The set $\CC(\Sigma,\PP)$ is endowed with a structure of simplicial complex, where a non-empty finite subset $\FF\subset\CC(\Sigma,\PP)$ is a simplex if $i([a],[b])=0$ for all $[a],[b]\in\FF$.
This complex is called the \emph{curve complex} of $(\Sigma,\PP)$.

In the present paper the right Dehn twist along a circle $a$ is denoted by $T_a$.

An \emph{arc} of $(\Sigma,\PP)$ is the image of an embedding $a\colon [0,1]\hookrightarrow\Sigma\setminus\partial\Sigma$ such that $a([0,1])\cap\PP=\{a(0),a(1)\}$, with $a(0)\ne a(1)$. We consider arcs up to isotopies under which interior points of arcs are mapped into $\Sigma\setminus(\partial\Sigma\cup\PP)$. In particular, such isotopies are required to fix the set $\PP$ pointwise.
We denote the isotopy class of an arc $a$ by $[a]$. 

With an arc $a$ of $(\Sigma,\PP)$ we associate an element $H_a\in\MM(\Sigma,\PP)$, called the \emph{(right) half-twist} along $a$.
This element is the identity outside a regular neighborhood of $a$, it exchanges the two ends of $a$, and $H_a^2=T_{\hat a}$, where $\hat a$ is the boundary of a regular neighborhood of $a$, see Figure~\ref{fig:halftwist}. We refer the reader to~\cite[Section~1.6.2]{KasTur1} for more information about properties of half-twists.

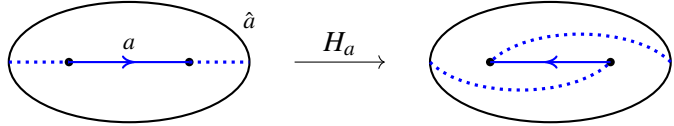
\begin{figure}[htb!]
\begin{center}
\begin{tikzpicture}[scale=0.8]
%\draw [help lines] (-2,-1) grid (3,2);
\draw[thick] (0,0) ellipse [x radius=2, y radius=1];
\fill (-1,0) circle (2pt);
\fill (1,0) circle (2pt);
\draw[dotted, very thick, blue] (-2,0)--(-1,0);
\draw[dotted, very thick, blue] (1,0)--(2,0);
\draw[->-=0.55,thick, blue] (-1,0)--(1,0);
\draw (0,0.3) node {\footnotesize$a$};
\draw (2,0.7) node {\footnotesize$\hat a$};

\begin{scope}[xshift=7cm]
\draw[thick] (0,0) ellipse [x radius=2, y radius=1];
\fill (-1,0) circle (2pt);
\fill (1,0) circle (2pt);
\draw[dotted, very thick, blue] (2,0) to [out=135,in=45,looseness=0.75] (-1,0); 
\draw[dotted, very thick, blue] (-2,0) to [out=-45,in=-135,looseness=0.75] (1,0);
\draw[->-=0.55,thick, blue] (1,0)--(-1,0);
\end{scope}

\draw[->] (2.75,0) to (4.25	,0); 
\draw (3.5,0.3) node {$H_a$};

\end{tikzpicture}
\end{center}
\caption{A half-twist\label{fig:halftwist}}
\end{figure}

We denote by $\D$ the closed disk, and we choose a collection $\PP_{n+1}=\{p_0,p_1,\dots, p_{n-1},p_{n}\}$ of $n+1$ punctures in the interior of $\D$ (see Figure~\ref{fig3_2}). Let $b_0$, $a_1,\dots,a_{n-1},a_{n}$ be the arcs drawn in Figure~\ref{fig3_2}. 
In what follows we will be using the following facts, which were established in~\cite{ParSor1}.

\begin{prop}\label{prop_mcg}
With the notation of Figure~\ref{fig3_2},
\begin{enumerate}
\item there is an isomorphism $\Psi\colon A[A_{n}]\to\MM(\D,\PP_{n+1})$ which sends $s_i$ to the half-twist $H_{a_i}$ for all $1\le i\le n$;
\item the element $t_0$ is identified via $\Psi$ with the half-twist $H_{b_0}$, so that the subgroup $A[\tilde A_{n-1}]=\langle t_0,\dots,t_{n-1}\rangle$ is generated by half-twists $H_{b_0}$, $H_{a_1}$, \dots, $H_{a_{n-1}}$;
\item the elements $u_0$, $u_1$, $v_0$, $v_1$ (appearing in Proposition~\ref{prop3_1}) are identified via $\Psi$ with half-twists $H_{b_0}$, $H_{b_1}$, $H_{c_0}$, $H_{c_1}$, respectively, where arcs $b_0$, $b_1$, $c_0$, $c_1$ are depicted in Figure~\ref{fig3_3};
\item the element $\Delta_Y^2$ corresponds via $\Psi$ to the Dehn twist $T_d$, where the circle $d$ is depicted in Figure~\ref{fig3_3}, and the element $\Delta^2$ corresponds to the Dehn twist $T_{\partial\D}$ about the boundary $\partial\D$ of the disk $\D$;
\item the subgroup $A[B_n]=\langle r_1,\dots,r_{n-1},r_n\rangle=\langle s_1,\dots,s_{n-1},s_n^2\rangle$ is identified via $\Psi$ with the stabilizer in $\MM(\D,\PP_{n+1})$ of the last puncture $p_{n}$, i.e.\ $A[B_n]=\{\,f\in\MM(\D,\PP_{n+1})\mid f(p_n)=p_n\}$.\pushQED{\qed}\qedhere\popQED
\end{enumerate}
\end{prop}

\begin{figure}[ht!]
\begin{center}
\begin{tikzpicture}[very thick, scale=1] 
\begin{scope}
\draw[blue] (0,0)--(0,1)--(0.67,0.67)--(1,0)--(0.83,-0.37);
\draw[blue] (-0.67,0.67)--(-1,0)--(-0.67,-0.67)--(-0.3,-0.85);
\draw[blue] (0.22,0.4) node {$\scriptstyle a_{n}$};
\draw[blue] (0.5,0.98) node {$\scriptstyle a_{n\mhyphen1}$};
\draw[blue] (1.1,0.4) node {$\scriptstyle a_{n\mhyphen2}$};
\draw[blue] (-1,0.4) node {$\scriptstyle a_{1}$};
\draw[blue] (-1,-0.4) node {$\scriptstyle a_{2}$};
\draw[red] (0,1)--(-0.67,0.67);
\draw[red] (-0.45,1) node {$\scriptstyle b_0$};
\draw[loosely dotted, very thick, blue] (-0.3,-0.85) to [out=330,in=240,looseness=0.75] (0.83,-0.37); 
\fill (0,0) circle (2pt) node [below right=-3pt] {$\scriptscriptstyle n$};
\fill (1,0) circle (2pt) node [right=-1pt] {$\scriptscriptstyle n\mhyphen3$};
\fill (0,1) circle (2pt) node [above=0pt] {$\scriptscriptstyle n\mhyphen1$};
\fill (-1,0) circle (2pt) node [left=-2pt] {$\scriptscriptstyle 1$};
\fill (0.67,0.67) circle (2pt) node [above right=-2pt] {$\scriptscriptstyle n\mhyphen2$};
\fill (-0.67,0.67) circle (2pt) node [above left=-3pt] {$\scriptscriptstyle 0$};
\fill (-0.67,-0.67) circle (2pt) node [below left=-3pt] {$\scriptscriptstyle 2$};
\draw[very thick] (0,0) ellipse (2cm and 1.4cm);
\end{scope} 
\end{tikzpicture}
\caption{Disk $\D$ with punctures $p_i$ (denoted for the ease of notation by $i$), $0\le i\le n$.
\label{fig3_2}}
\end{center}
\end{figure}
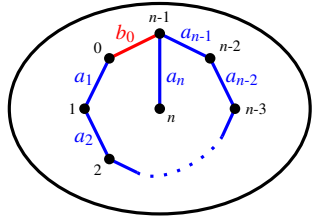

\begin{figure}[ht!]
\begin{center}
\includegraphics[width=4cm]{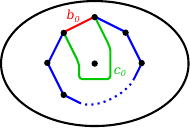}\qquad
\includegraphics[width=4cm]{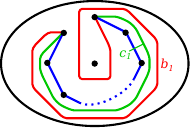}\qquad
\includegraphics[width=4cm]{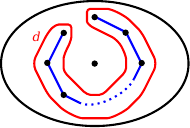}
\caption{Arcs and a circle in the punctured disk for Proposition~\ref{prop_mcg}}\label{fig3_3}
\end{center}
\end{figure}

Let $\Sigma$ be an oriented compact surface, and let $\PP$ be a finite collection of punctures in the interior of $\Sigma$.
Assume the Euler characteristic of $\Sigma\setminus\PP$ is negative.
Let $f\in\MM(\Sigma,\PP)$.
We say that a simplex $\FF$ of $\CC(\Sigma,\PP)$ is a \emph{reduction system} for $f$ if $f(\FF)=\FF$.
In this case any element of $\FF$ is called a \emph{reduction class} for $f$.
A reduction class $[a]$ is an \emph{essential reduction class} if, for each $[b]\in\CC(\Sigma,\PP)$ such that $i([a],[b])\neq 0$ and for each $m\in\Z\setminus\{0\}$, we have $f^m([b])\neq [b]$.
In particular, if $[a]$ is an essential reduction class and $[b]$ is any reduction class, then $i([a],[b])=0$.
We denote by $\SS(f)$ the set of essential reduction classes for $f$.
The following gathers together some results on $\SS(f)$ that we will use in the proofs.

\begin{thm}[Birman--Lubotzky--McCarthy~\cite{BiLuMc1}]\label{thm3_3}
Let $\Sigma$ be an oriented compact surface, and let $\PP$ be a finite collection of punctures in the interior of $\Sigma$. 
Assume the Euler characteristic of $\Sigma\setminus\PP$ is negative.
Let $f\in\MM(\Sigma,\PP)$.
\begin{itemize}
\item[(1)]
If $\SS(f)\neq\varnothing$, then $\SS(f)$ is a reduction system for $f$.
\item[(2)]
We have $\SS(f^n)=\SS(f)$ for all $n\in\Z\setminus\{0\}$.
\item[(3)]
We have $\SS(gfg^{-1})=g(\SS(f))$ for all $g\in\MM(\Sigma,\PP)$.\qed
\end{itemize}
\end{thm}

In addition to Theorem~\ref{thm3_3} we have the following proposition, which is well-known and which is a direct consequence of~\cite{BiLuMc1} (see also~\cite[Corollaire~3.45]{Caste1}).

\begin{prop}\label{prop3_4}
Let $\Sigma$ be an oriented compact surface, and let $\PP$ be a finite collection of punctures in the interior of $\Sigma$.
Assume the Euler characteristic of $\Sigma\setminus\PP$ is negative.
Let $f_0\in Z(\MM(\Sigma,\PP))$ be a central element of $\MM(\Sigma,\PP)$, let $\FF=\{[a_1],[a_2],\dots,[a_p]\}$ be a simplex of $\CC(\Sigma,\PP)$, and let $k_1,k_2,\dots,k_p$ be non-zero integers.
Let $g=T_{a_1}^{k_1}\,T_{a_2}^{k_2}\dots T_{a_p}^{k_p}\,f_0$.
Then $\SS(g)=\FF$.\qed
\end{prop}

%%%%%%%%%%

\section{Proofs for endomorphisms of \texorpdfstring{$A[B_{n}]$}{A[Bn]}}\label{sec4}

\begin{proof}[Proof of Theorem~\ref{thm2_1}]
Let $n\ge5$ and $\varphi\colon A[B_{n}]\longrightarrow A[B_{n}]$ be an endomorphism. 
Pre-composing $\varphi$ with the embedding $\iota_{\tilde A}\colon A[\tilde A_{n-1}]\lhook\joinrel\longrightarrow A[B_{n}]$, and post-com\-po\-sing it with the embedding $\iota_B\colon A[B_{n}]\lhook\joinrel\longrightarrow A[A_{n}]$, we get a homomorphism $\psi\colon A[\tilde A_{n-1}]\longrightarrow A[A_{n}]$:
\[
\psi\colon A[\tilde A_{n-1}] \stackrel{\iota_{\tilde A}}{\lhook\joinrel\longrightarrow} A[B_{n}] \stackrel{\varphi}{\longrightarrow} A[B_{n}] \stackrel{\iota_{B}}{\lhook\joinrel\longrightarrow} A[A_{n}]. 
\] 
By Proposition~\ref{prop3_1}, we have one of the three possibilities:
\begin{itemize}
\item[(1)] $\psi$ is cyclic.
\item[(2)] There exist $g\in A[A_{n}]$, $k\in\{0,1\}$, $\varepsilon\in\{\pm1\}$ and $p\in\Z$ such that $\psi(t_i)=g\,t_i^\varepsilon\,\Delta^{2p}\,g^{-1}$ for all $1\le i\le n-1$, and $\psi(t_0)=g\,u_k^\varepsilon\,\Delta^{2p}\,g^{-1}$.
\item[(3)] There exist $g\in A[A_{n}]$, $k\in\{0,1\}$, $\varepsilon\in\{\pm1\}$ and $p,q\in\Z$ such that $\psi(t_i)=g\,t_i^\varepsilon\,\Delta_Y^{2p}\,\Delta^{2q}\,g^{-1}$ for all $1\le i\le n-1$, and $\psi(t_0)=g\,v_k^\varepsilon\,\Delta_Y^{2p}\,\Delta^{2q}\,g^{-1}$.
\end{itemize}
Here $Y=\{t_1,\dots,t_{n-1}\}$, $\Delta$ and $\Delta_Y$ denote the Garside elements of $A[A_n]$ and $A_Y[A_{n}]=A_Y[\tilde A_{n-1}]$, respectively (so that $\Delta^2=\Delta_B$), $u_0=t_0$, $u_1=\Delta_Y^{-1}t_0\Delta_Y$, $v_0=\rho_0 t_{n-1}\rho_0^{-1}$, $v_1=\rho_1 t_{n-1}\rho_1^{-1}$, where $\rho_0=t_1\dots t_{n-1}$ and $\rho_1=t_1^{-1}\dots t_{n-1}^{-1}$.

We analyze all three possibilities (1), (2), (3) and deduce a description of $\varphi$ in each of these cases.

Case (1): $\psi$ is cyclic. Then, since $\iota_B$ is injective, $\varphi\circ \iota_{\tilde A}\colon A[\tilde A_{n-1}]\to A[B_{n}]$ is cyclic as well. 
From the relations $t_i t_{i+1} t_i = t_{i+1} t_i t_{i+1}$ (with indices taken modulo $n$), it follows that there exists $g \in A [B_n]$ such that $\varphi (t_i) = g$ for all $i$. 
Let $h = \varphi (\rho_B)$. Since $\rho_B t_i \rho_B^{-1}  = t_{i+1}$, we see that $h g h^{-1} = \varphi (\rho_B)\, \varphi (t_i)\, \varphi(\rho_B)^{-1} = \varphi (t_{i+1}) = g$, hence $h$ commutes with $g$,
and $\varphi$ has the form: 
\[
\varphi\colon A[\tilde A_{n-1}]\rtimes \langle \rho_B\rangle\to A[B_{n}],\quad t_i\mapsto g,\quad \rho_B\mapsto h.
\]
On the other hand, any two commuting elements $g,h\in A[B_{n}]$ define such an endomorphism of $A[B_{n}]$, so we get case~(1) of Theorem~\ref{thm2_1}.

Case (2): We have $\psi(t_i)=g\,t_i^\varepsilon\,\Delta^{2p}\,g^{-1}$ for all $1\le i\le n-1$, and $\psi(t_0)=g\,u_k^\varepsilon\,\Delta^{2p}\,g^{-1}$ for some $p\in\Z$, $g\in A[A_{n}]$, $\varepsilon\in\{\pm 1\}$, $k\in\{0,1\}$. Arguing as in the proof of Theorem~2.1 of~\cite{ParSor1}, we see that $g(p_n)=p_n$, and hence $g$ belongs to $A[B_{n}]$, and we can consider $\varphi'=\conj_{g^{-1}}\circ \varphi\colon A[B_{n}]\to A[B_{n}]$, which acts on elements $t_i$ as follows:
\begin{align*}
\varphi'(t_i)&=t_i^\varepsilon\,\Delta^{2p},\quad\text{for\,\, }1\le i\le n-1,\\
\varphi'(t_0)&=u_k^\varepsilon\,\Delta^{2p}.
\end{align*}
We need to figure out how $\varphi'$ acts on $\rho_B$. Consider two cases: $k=0$ and $k=1$.

Let $k=0$. 
Then $\varphi' (t_0) = u_0^\varepsilon \Delta^{2p} = t_0^\varepsilon \Delta^{2p}$. We can write $\varphi'(\rho_B)=h\cdot\rho_B$ for some $h\in A[B_{n}]$. We are going to prove that $h\in Z(A[B_{n}])$. Indeed, the equality $\varphi'(\rho_B\, t_i\,\rho_B^{-1})=\varphi'(t_{i+1})$ reads for $0\le i\le n-1$ (with indices taken modulo $n$):
\[
h\cdot\rho_B\cdot t_i^\varepsilon \Delta^{2p}\cdot\rho_B^{-1}\cdot h^{-1} = t_{i+1}^\varepsilon\Delta^{2p},
\]
which yields
\[
h \cdot t_{i+1}^\varepsilon\cdot h^{-1}=t_{i+1}^\varepsilon, \quad \text{for}\quad 1\le i+1\le n.
\]
Hence $h$ commutes with all generators of $A[\tilde A_{n-1}]$, and the following lemma shows that $h$ belongs to the center of $A[B_{n}]$, which is generated by $\Delta_B$.

\begin{lem}\label{lem4_1}
$Z_{A[B_{n}]}(A[\tilde A_{n-1}]) = Z(A[B_{n}])$.
\end{lem}
\begin{proof}
Let $z\in Z_{A[B_{n}]}(A[\tilde A_{n-1}])$. Then $z$ can be written according to the semidirect structure of $A[B_{n}]$ as $z=h\cdot\rho_B^k$ for some $h\in A[\tilde A_{n-1}]$ and $k\in\Z$. Since $z t_i z^{-1} = t_i$ for all $0\le i\le n-1$, we have
\[
t_i=h\cdot\rho_B^k\cdot t_i\cdot \rho_B^{-k}\cdot h^{-1}=
h\cdot t_{i+k}\cdot h^{-1},
\]
or $h^{-1}\cdot t_i\cdot h=t_{i+k}$. In particular, $h^{-n}\cdot t_i\cdot h^{n}=t_{i+kn}=t_i$ for all $i$. Hence $h^{n}$ belongs to the center of $A[\tilde A_{n-1}]$, which is trivial. Since $A[\tilde A_{n-1}]$ is torsion free, we conclude that $h=1$, $z=\rho_B^k$ and $t_{i+k}=t_i$ for all $i$. This means that $k$ is a multiple of $n$, and hence $z\in\langle \rho_B^{n}\rangle=Z(A[B_{n}])$.
\end{proof}

We conclude that in Case (2) with $k=0$, the homomorphism $\varphi'$ has the form:
\begin{align*} 
\begin{split}
\varphi'(t_i)&=t_i^\varepsilon\Delta^{2p}=t_i^\varepsilon\Delta_B^{p},\quad\text{for\,\, }0\le i\le n-1,\\
\varphi'(\rho_B)&=\rho_B\Delta^{2q}=\rho_B\Delta_B^{q},\,\,\text{ for some\,\, }p,q\in \Z.
\end{split}
\end{align*}
and this gives us the endomorphism of type (2a) of Theorem~\ref{thm2_1}.
 
Let $k=1$ now. Then for $1\le i\le n-1$ we can write $t_i=\Delta_Y^{-1\,}t_{n-i}\,\Delta_Y$. Then
\begin{align*}
\varphi'(t_i)&=\Delta_Y^{-1}\,t_{n-i}^\varepsilon\, \Delta_Y\cdot\Delta^{2p},\quad\text{for\,\, }1\le i\le n-1,\\
\varphi'(t_0)&=u_1^\varepsilon\cdot\Delta^{2p}=\Delta_Y^{-1}\,t_0^\varepsilon\, \Delta_Y\cdot\Delta^{2p}.
\end{align*}
Denote $\varphi''=\conj_{\Delta_Y}\circ \varphi'$. Then 
\begin{align*}
\varphi''(t_i)&=t_{n-i}^\varepsilon\cdot\Delta^{2p},\quad\text{for\,\, }1\le i\le n-1,\\
\varphi''(t_0)&=t_0^\varepsilon\cdot\Delta^{2p}.
\end{align*}
Let us write $\varphi''(\rho_B)=h\cdot \rho_B^{-1}$ for some $h\in A[B_{n}]$. Then the equality $\varphi''(\rho_B\, t_i\, \rho_B^{-1})=\varphi''(t_{i+1})$ yields: $h\rho_B^{-1}\cdot t_{n-i}^\varepsilon\cdot \rho_Bh^{-1}=t_{n-1-i}^\varepsilon$ for all $i$, which means that $h\cdot t_{n-1-i}^\varepsilon\cdot h^{-1}=t_{n-1-i}^\varepsilon$ for all $0\le i\le n-1$. Therefore $h\in Z_{A[B_{n}]}(A[\tilde A_{n-1}])$, and, by Lemma~\ref{lem4_1}, $h=\Delta^{2q}$ for some $q\in\Z$, so that $\varphi''(\rho_B)=\rho_B^{-1}\,\Delta^{2q}=\rho_B^{-1}\,\Delta_B^q$, and we get an endomorphism of type (2b) of Theorem~\ref{thm2_1}.

Case (3): We have $\psi(t_i)=g\,t_i^\varepsilon\,\Delta_Y^{2p}\,\Delta^{2q}\,g^{-1}$ for all $1\le i\le n-1$, and $\psi(t_0)=g\,v_k^\varepsilon\,\Delta_Y^{2p}\,\Delta^{2q}\,g^{-1}$ for some $p,q\in\Z$, $g\in A[A_{n}]$, $\varepsilon\in\{\pm 1\}$, $k\in\{0,1\}$. Again, arguing as in the proof of Theorem~2.1 of~\cite{ParSor1}, we see that $g$ belongs to $A[B_{n}]$, and we can consider $\varphi'=\conj_{g^{-1}}\circ \varphi\colon A[B_{n}]\to A[B_{n}]$, which acts on elements $t_i$ as follows:
\begin{align*}
\varphi'(t_i)&=t_i^\varepsilon\,\Delta_Y^{2p}\,\Delta^{2q},\quad\text{for\,\, }1\le i\le n-1,\\
\varphi'(t_0)&=v_k^\varepsilon\,\Delta_Y^{2p}\,\Delta^{2q}.
\end{align*}
To determine $\varphi'(\rho_B)$ we write $\varphi'(\rho_B)=h\cdot\rho_k$ for some $h\in A[B_{n}]$. Then the condition $\varphi'(\rho_B t_i\rho_B^{-1})=\varphi'(t_{i+1})$ gives us: 
\begin{align*}
h\rho_k&\cdot t_i^\varepsilon\Delta_Y^{2p}\cdot\rho_k^{-1} h^{-1}=t_{i+1}^\varepsilon\cdot\Delta_Y^{2p},\quad\text{for\,\, $1\le i\le n-2$},\\
h\rho_k&\cdot \rho_k t_{n-1}^\varepsilon\rho_k^{-1}\Delta_Y^{2p}\cdot\rho_k^{-1} h^{-1}=t_{1}^\varepsilon\cdot\Delta_Y^{2p}.
\end{align*}
Notice that $\rho_k$ commutes with $\Delta^2_Y$, hence the above equalities read:
\begin{align*}
h\, t_{i+1}^\varepsilon\, h^{-1}&\cdot h\,\Delta_Y^{2p}\,h^{-1}=t_{i+1}^\varepsilon\cdot \Delta_Y^{2p},\quad\text{for\,\, $2\le i+1\le n-1$},\\
h\, t_1^\varepsilon\, h^{-1}&\cdot h\,\Delta_Y^{2p}\,h^{-1}=t_{1}^\varepsilon\cdot\Delta_Y^{2p}.
\end{align*}
Recall that elements $t_i$ ($1\le i\le n-1$) can be viewed as half-twists about arcs $a_i$ depicted in Figure~\ref{fig3_2}, and we denote as before by $\hat a_i$ the boundary of a regular neighborhood of $a_i$. Recall also that $\Delta_Y^2=T_d$ is the Dehn twist about the circle $d$, depicted in Figure~\ref{fig3_3}. We now consider the sets of essential reduction classes for the elements defined by the equalities above, viewing them as mapping classes from $\MM(\D,\PP_{n+1})$. Recall that, in our notation, $\SS(f)$ denotes the set of all essential reduction classes for a mapping class $f$.

If $p=0$, we have 
\[
h\, t_i^\varepsilon\, h^{-1}=t_{i}^\varepsilon,\quad \text{for }1\le i\le n-1,
\]
and we notice that $\SS(t_i^\varepsilon)=\SS(t_i^{2\varepsilon})$ by Theorem~\ref{thm3_3}(2). Since $t_i^2$ is the Dehn twist $T_{\hat a_i}$ about $\hat a_i$, we conclude that $\SS(t_i^{2\varepsilon})=\{[\hat a_i]\}$ by Proposition~\ref{prop3_4}, for $1\le i\le n-1$. By Theorem~\ref{thm3_3}(3), we see that $\SS(h\, t_i^\varepsilon\,h^{-1})=h\SS(t_i^\varepsilon)=\{h[\hat a_i]\}$, and hence $h[\hat a_i]=[\hat a_i]$ for all $1\le i\le n-1$.

If $p\ne0$, we have 
\[
h\, t_{i}^\varepsilon\, \Delta_Y^{2p}\,h^{-1}=t_{i}^\varepsilon\, \Delta_Y^{2p},\quad\text{for\,\, $1\le i\le n-1$},
\]
and we notice that $\SS(t_i^\varepsilon\Delta_Y^{2p})=\SS(t_i^{2\varepsilon}\Delta_Y^{4p})$ by Theorem~\ref{thm3_3}(2). Since $t_i^{2\varepsilon}\Delta_Y^{4p}=T_{\hat a_i}^\varepsilon T_d^{2p}$, Proposition~\ref{prop3_4} tells us that $\SS(t_i^{2\varepsilon}\Delta_Y^{4p})=\{[\hat a_i],[d]\}$. Now, the above equality 
implies by Theorem~\ref{thm3_3}(3) that $h\{[\hat a_i],[d]\}=\{[\hat a_i],[d]\}$ for all $1\le i\le n-1$. Since $d$ encloses more punctures than any of the circles $\hat a_i$, we conclude that $h[d]=[d]$ and $h[\hat a_i]=[\hat a_i]$ for $1\le i\le n-1$, as in the case $p=0$ above.

Let $\D_i\subset \D$ denote the disk bounded by the circle $\hat a_i$. Since $h[\hat a_i]=[\hat a_i]$ for each $i$, we can choose a representative $H_i\in\Homeo^+(\D,\PP_{n+1})$ of the mapping class $h$ such that 
$H_i (\D_i) = \D_i$ and $H_i|_{\partial \D_i} = \id_{\partial \D_i}$.
Then $[H_i|_{\D_i}]$ belongs to $\MM(\D_i,\PP_2)$, which is a cyclic group generated by the half-twist $H_{a_i}=t_i$. Thus, $[H_i|_{\D_i}]$ commutes with $t_i$, and hence $h$ commutes with elements $t_i$ for all $1\le i\le n-1$. By~\cite[Theorem~1.1]{Paris2}, $h\in\langle \Delta_Y^2,\Delta^2\rangle$, so that $\varphi'(\rho_B)=\Delta_Y^{2r}\Delta^{2s}\rho_k=\rho_k\Delta_Y^{2r}\Delta_B^{s}$, for some $r,s\in \Z$. Therefore $\varphi'$ is an endomorphism of type~(3) of Theorem~\ref{thm2_1}.
\end{proof}

\begin{proof}[Proof of Proposition~\ref{prop2_1}]
The type of an endomorphism $\varphi$ from Theorem~\ref{thm2_1} can be characterized by the following properties, according to Corollary~\ref{cor2_4} (which we are going to prove later in this section, without reliance on the proof of Proposition~\ref{prop2_1}): if $\varphi$ is of type (1), then the image of $\varphi$ is abelian and $\varphi$ is non-injective; if $\varphi$ is of type (2a) or (2b), then $\varphi$ is injective; and if $\varphi$ is of type (3), then the image of $\varphi$ is non-abelian and $\varphi$ is non-injective. Notice that if $\varphi=\conj_x\circ\psi$, then $\psi$ has the same classification according to these properties as $\varphi$, since $\conj_x$ is an automorphism. Thus, it remains to distinguish types (2a) and (2b) up to conjugation, and also to prove that parameters $\varepsilon,k,p,q,r,s$ are determined uniquely. For that, consider a homomorphism 
\[
\eta\colon A[B_n]=A[\tilde A_{n-1}]\rtimes\langle\rho_B\rangle\longrightarrow\Z^2,\quad t_i\longmapsto (1,0),\quad \rho_B\longmapsto (0,1).
\]
Notice that $\eta\circ\conj_x\circ\varphi = \eta\circ\varphi$ for any $\varphi$ and $x$, since the image of $\eta$ is an abelian group. Thus it will suffice to show that for an arbitrary endomorphism $\varphi$ of type (2a), (2b) or (3) of Theorem~\ref{thm2_1}, its type and the corresponding parameters are uniquely determined by the composition $\eta\circ\varphi$.

Observe that $\eta(\Delta_B)=\eta(\rho_B^n)=(0,n)$ and $\eta(\Delta_Y^2)=(n(n-1),0)$. Notice also that $\eta(v_k)=\eta(\conj_{\rho_k}(t_{n-1}))=\eta(t_{n-1})=(1,0)$. Hence we have the following formulas for $\eta\circ\varphi$ for different types of $\varphi$ ($0\le i\le n-1$):
\[
\begin{aligned}
\text{type (2a):\qquad}\eta\circ\varphi:&\quad t_i\longmapsto(\varepsilon,pn),\quad \rho_B\longmapsto(0,qn+1);\\
\text{type (2b):\qquad}\eta\circ\varphi:&\quad t_i\longmapsto(\varepsilon,pn),\quad \rho_B\longmapsto(0,qn-1);\\
\text{type (3), $k=0$:\qquad}\eta\circ\varphi:&\quad t_i\longmapsto(\varepsilon+pn(n-1),qn),\quad \rho_B\longmapsto
(n-1+rn(n-1),sn);\\
\text{type (3), $k=1$:\qquad}\eta\circ\varphi:&\quad t_i\longmapsto(\varepsilon+pn(n-1),qn),\quad \rho_B\longmapsto
(-(n-1)+rn(n-1),sn).
\end{aligned}
\]
In particular, we see that an endomorphism $\varphi$ cannot belong to both types (2a) and (2b). Indeed, assume that $\varphi$ corresponds simultaneously to parameters $q=q_1$ and $q=q_2$, i.e.\ $\varphi(\rho_B)=\rho_B\Delta_B^{q_1}=x\rho_B^{-1}\Delta_B^{q_2}x^{-1}$ for some $x\in A[B_n]$, then we have from the above formulas that $q_1n+1=q_2n-1$, which is equivalent to $(q_2-q_1)n=2$. This equality is impossible since $n>2$. It is also clear that parameters $\varepsilon,p,q$ are uniquely determined in cases (2a) and (2b).

Similarly, $\varphi$ cannot simultaneously belong to type (3) with both values of $k=0,1$. Indeed, if this is the case, then for some $r_1,r_2$ we have $n-1 + r_1 n(n-1)=-(n-1)+r_2n(n-1)$, which is equivalent to $(r_2-r_1)n=2$. The latter equality is impossible since $n>2$.

We see in the same way that parameters $\varepsilon$ and $p$ are determined uniquely for $\varphi$ of type (3). Indeed, suppose that there are some $\varepsilon_1,\varepsilon_2\in\{\pm1\}$ and some $p_1,p_2\in \Z$ such that $\varepsilon_1+p_1n(n-1)=\varepsilon_2+p_2n(n-1)$, which is equivalent to $(p_2-p_1)n(n-1)=\varepsilon_1-\varepsilon_2$, with $\varepsilon_1-\varepsilon_2\in\{0,\pm2\}$. Since $n>2$, this equality is only possible if $p_2-p_1=\varepsilon_1-\varepsilon_2=0$.

Similarly, parameter $r$ is determined uniquely for $\varphi$ of type (3). Indeed, assume that for some $r_1,r_2$ we have $\pm(n-1)+r_1n(n-1)=\pm(n-1)+r_2n(n-1)$ (with the same choice of signs on both sides). Then if follows that $r_1=r_2$. Also, it is clear that parameters $q$ and $s$ are determined uniquely for $\varphi$ of type (3).
\end{proof}

\begin{proof}[Proof of Corollary~\ref{cor2_3}]
Assume that we have an endomorphism (1) of Theorem~\ref{thm2_1}, i.e.\ there exist $g,h\in A[B_{n}]$ such that $gh=hg$, $\varphi(t_i)=g$ for $0\le i\le n-1$, and $\varphi(\rho_B)=h$. To write $\varphi$ in the standard generators $r_1,\dots, r_{n-1}, r_{n}$ of $A[B_{n}]$, we observe that $t_i=r_i$, $1\le i\le n-1$, $\rho_B=t_1\dots t_{n-1} r_{n}$, and $\varphi(\rho_B)=g^{n-1}\varphi(r_{n})=h$. Then $\varphi(r_{n})=g^{-n+1}h$ commutes with $g$. On the other hand, an arbitrary element $h$ of $A[B_{n}]$ commuting with $g$ defines an endomorphism $\varphi$ by setting $\varphi(r_i)=g$ and $\varphi(r_{n})=h$, which is an endomorphism of type (I).

Assume now that we have an endomorphism of type (2a) with $\varepsilon=1$ in Theorem~\ref{thm2_1}, i.e.\ $\varphi(t_i)=t_i\Delta_B^p$ for $0\le i\le n-1$ and $\varphi(\rho_B)=\rho_B\Delta_B^q$ for some $p,q\in\Z$. Then $\varphi(r_i)=r_i\Delta_B^{p}$ for $1\le i\le n-1$, and $\varphi(\rho_B)=\varphi(r_1)\dots\varphi(r_{n-1})\varphi(r_{n})=r_1\dots r_{n-1}\varphi(r_{n})\Delta_B^{p(n-1)}$, and hence
$r_1r_2\dots r_{n-1}\varphi(r_{n})\Delta_B^{p(n-1)}=r_1\dots r_{n-1}r_{n}\Delta_B^{q}$, which gives $\varphi(r_{n})=r_{n}\Delta_B^{q-p(n-1)}$. Since $q\in\Z$ can be arbitrary, we get an endomorphism of case (IIa) with $\varepsilon=1$.

For an endomorphism of type (2a) with $\varepsilon=-1$ we have: $\varphi(t_i)=t_i^{-1}\Delta_B^p$, for $0\le i\le n-1$, and $\varphi(\rho_B)=\rho_B\Delta_B^q$ for some $p,q\in\Z$. The last equality gives us:
\[
\varphi(\rho_B)=\varphi(r_1r_2\dots r_{n-1}r_{n})=r_1^{-1}r_2^{-1}\dots r_{n-1}^{-1}\varphi(r_{n})\Delta_B^{p(n-1)}=r_1r_2\dots r_{n-1}r_{n}\Delta_B^q,
\]
which is equivalent to $\varphi(r_{n})=r_{n-1}r_{n-2}\dots r_2\,(r_1^2)\,r_2\dots r_{n-1} r_{n}\Delta_B^{q-p(n-1)}=\delta\cdot r_{n}\Delta_B^{q-p(n-1)}$. Since $q\in\Z$ can be arbitrary, we get a description of an endomorphism of type (IIb) with $\varepsilon=-1$.

To analyze endomorphisms of type (2b), it will be more convenient to represent them differently. An endomorphism of type (2b) has the form: $\varphi(t_i)=t_{n-i}^\varepsilon\Delta_B^p$ and $\varphi(\rho_B)=\rho_B^{-1}\Delta_B^q$, for some $p,q\in\Z$. Notice that $\Delta_Y^{-1}t_{n-i}\Delta_Y=t_i$ for all $0\le i\le n-1$, if we consider indices $i$ modulo $n$. Let $\varphi'=\conj_{\Delta_Y^{-1}}\circ \varphi$, then the homomorphism $\varphi'$ can be written as:
\begin{align*} 
\begin{split}
\varphi'(t_i)&=t_i^\varepsilon\cdot\Delta_B^{p},\quad\text{for\,\, }0\le i\le n-1,\\
\varphi'(\rho_B)&=(\Delta_Y^{-1}\rho_B^{-1}\Delta_Y)\cdot\Delta_B^{q},\,\,\text{ for some\,\, }p,q\in \Z.
\end{split}
\end{align*}
We need the following lemma.

\begin{lem}\label{lem4_2}
$\Delta_Y^{-1}\rho_B\Delta_Y = r_{n}r_{n-1}\dots r_1$.
\end{lem}
\begin{proof}
Notice that $\rho_B=r_1\dots r_{n-1}r_{n}=\rho_0\cdot r_{n}$ and $r_{n}r_{n-1}\dots r_1=r_{n}\cdot(\rho_1)^{-1}$, hence we need to show that
\[
(\Delta_Y^{-1}\rho_0)\cdot r_{n}\cdot(\Delta_Y \rho_1)=r_{n}.
\]
Recall that $\Delta_Y=(r_1r_2\dots r_{n-1})(r_1r_2\dots r_{n-2})\dots(r_1r_2)r_1=r_1(r_2r_1)(r_3r_2r_1)\dots(r_{n-1}r_{n-2}\dots r_2r_1)$. Denote $Y_1=\{r_1,\dots,r_{n-2}\}$, and $\Delta_{Y_1}=r_1(r_2r_1)\dots(r_{n-2}\dots r_2r_1)$. We see that 
\[
\Delta_Y\cdot\rho_1=r_1(r_2r_1)\dots(r_{n-2}\dots r_2r_1)(r_{n-1}r_{n-2}\dots r_2r_1)\cdot \rho_1=r_1(r_2r_1)\dots(r_{n-2}\dots r_2r_1)=\Delta_{Y_1},
\]
and 
\[
\Delta^{-1}_Y\cdot\rho_0=r_1^{-1}(r_2^{-1}r_1^{-1})\dots(r_{n-2}^{-1}\dots r_2^{-1}r_1^{-1})(r_{n-1}^{-1}r_{n-2}^{-1}\dots r_2^{-1}r_1^{-1})\cdot \rho_0=\Delta_{Y_1}^{-1}. 
\]
Hence, 
\[
(\Delta_Y^{-1}\rho_0)\cdot r_{n}\cdot(\Delta_Y \rho_1)=\Delta_{Y_1}^{-1}\cdot r_{n}\cdot\Delta_{Y_1}=r_{n},
\]
since $r_{n}$ commutes with all generators from $\Delta_{Y_1}$.
\end{proof}

We conclude by Lemma~\ref{lem4_2}, that an endomorphism of case (2b) may be written as follows:
\begin{align*} 
\begin{split}
\varphi'(t_i)&=t_i^\varepsilon\cdot\Delta_B^{p},\quad\text{for\,\, }0\le i\le n-1,\\
\varphi'(\rho_B)&=r_1^{-1}\dots r_{n-1}^{-1}r_{n}^{-1}\cdot\Delta_B^{q},\,\,\text{ for some\,\, }p,q\in \Z.
\end{split}
\end{align*}

Thus in case (2b) with $\varepsilon=1$, we have:
\[
\varphi'(\rho_B)=\varphi'(r_1)\varphi'(r_2)\dots \varphi'(r_{n-1})\varphi'(r_{n})=r_1r_2\dots r_{n-1}\varphi'(r_{n})\Delta_B^{p(n-1)}=r_1^{-1}\dots r_{n-1}^{-1}r_{n}^{-1}\Delta_B^q,
\]
which is equivalent to 
\[
\varphi'(r_{n})=r_{n-1}^{-1}\dots r_2^{-1}(r_1^{-2})r_2^{-1}\dots r_{n-1}^{-1}\cdot r_{n}^{-1}\Delta_B^{q-p(n-1)}=\delta^{-1}\cdot r_{n}^{-1}\Delta_B^{q-p(n-1)}\,.
\]
Since $q\in\Z$ is arbitrary, we get an endomorphism of case (IIb) with $\varepsilon=1$.

Similarly, in case (2b) with $\varepsilon=-1$, we get:
\[
\varphi'(\rho_B)=\varphi'(r_1)\varphi'(r_2)\dots \varphi'(r_{n-1})\varphi'(r_{n})=r_1^{-1}r_2^{-1}\dots r_{n-1}^{-1}\varphi'(r_{n})\Delta_B^{p(n-1)}=r_1^{-1}\dots r_{n-1}^{-1}r_{n}^{-1}\Delta_B^q,
\]
which is equivalent to $\varphi'(r_{n})=r_{n}^{-1}\cdot\Delta_B^{q-p(n-1)}$. Since $q\in\Z$ is arbitrary, we get an endomorphism of case (IIa) with $\varepsilon=-1$.

Assume now that we have an endomorphism of type (3) with $k=0$ and $\varepsilon=1$ in Theorem~\ref{thm2_1}, i.e.\ $\varphi (t_i)=t_i\Delta_Y^{2p}\Delta_B^q=r_i\Delta_Y^{2p}\Delta_B^q$, for $1\le i\le n-1$, and $\varphi(\rho_B)=\rho_0\Delta_Y^{2r}\Delta_B^s=r_1r_2\dots r_{n-1}\Delta_Y^{2r}\Delta_B^s$ for some $p,q,r,s\in\Z$. Since $\rho_B=r_1\dots r_{n-1} r_{n}$, we see that
\[
\varphi (\rho_B)=\varphi (r_1)\dots \varphi (r_{n-1})\varphi (r_{n})= r_1 r_2\dots r_{n-1}\Delta_Y^{2p(n-1)}\Delta_B^{q(n-1)}\varphi (r_{n})=r_1r_2\dots r_{n-1}\Delta_Y^{2r}\Delta_B^s,
\]
from what we deduce that $\varphi (r_{n})=\Delta_Y^{2r-2p(n-1)}\Delta_B^{s-q(n-1)}$. Since $r$ and $s$ are arbitrary integers, we get an endomorphism of type (IIIa) with $\varepsilon=1$. Similarly, an endomorphism of type (3) with $k=1$ and $\varepsilon=-1$ of Theorem~\ref{thm2_1} yields an endomorphism of type (IIIa) with $\varepsilon=-1$.

An endomorphism of type (3) with $k=0$ and $\varepsilon=-1$ in Theorem~\ref{thm2_1} is given by the formulas: $\varphi(t_i)=t_i^{-1}\Delta_Y^{2p}\Delta_B^q=r_i^{-1}\Delta_Y^{2p}\Delta_B^q$, for $1\le i\le n-1$, and $\varphi (\rho_B)=\rho_0\Delta_Y^{2r}\Delta_B^s=r_1r_2\dots r_{n-1}\Delta_Y^{2r}\Delta_B^s$ for some $p,q,r,s\in\Z$. On the other hand, 
\[
\varphi (\rho_B)=\varphi (r_1)\dots\varphi (r_{n-1})\varphi (r_{n})=r_1^{-1}\dots r_{n-1}^{-1}\Delta_Y^{2p(n-1)}\Delta_B^{q(n-1)}\varphi (r_{n}), 
\]
from what we get that 
\[
\varphi (r_{n})=r_{n-1} r_{n-2}\dots r_2(r_1^2)r_2\dots r_{n-2}r_{n-1}\Delta_Y^{2r-2p(n-1)}\Delta_B^{s-q(n-1)}=\delta\cdot\Delta_Y^{2r-2p(n-1)}\Delta_B^{s-q(n-1)}.
\]
Since $r$ and $s$ are arbitrary integers, we get an endomorphism of type (IIIb) with $\varepsilon=-1$.

In exactly the same manner, an endomorphism of type (3) with $k=1$ and $\varepsilon=1$ in Theorem~\ref{thm2_1} yields an automorphism of type (IIIb) with $\varepsilon=1$. This finishes the proof of Corollary~\ref{cor2_3}.
\end{proof}

\begin{proof}[Proof of Corollary~\ref{cor2_4}]
The endomorphisms of type (1) of Theorem~\ref{thm2_1} are neither injective nor surjective, since their image is an abelian subgroup of $A[B_{n}]$.
 
The endomorphisms of type (3) are not injective. To see this we notice by direct check that if $k=0$ and $\varepsilon=1$ or if $k=1$ and $\varepsilon=-1$ we have $\varphi(t_0)=\varphi(v_0)$. Similarly, if $k=0$ and $\varepsilon=-1$ or if $k=1$ and $\varepsilon=1$ we have $\varphi(t_0)=\varphi(v_1)$. Since $t_0$, $v_0$ and $v_1$ are half-twists about pairwise non-isotopic arcs $b_0$, $c_0$, and $c_1$, respectively, in $\MM(\D,\PP_{n+1})$ (see Figure~\ref{fig3_3}), we conclude by Corollary~3.6 of~\cite{ParSor1} that $t_0$, $v_0$ and $v_1$ are three distinct elements of $\MM(\D,\PP_{n+1})$, and hence $\varphi$ is not injective.

To see that an endomorphism $\varphi$ of type (3) is not surjective, we show that its composition $\pi\circ\varphi$ with the natural epimorphism by the center $\pi\colon A[B_{n}]\to A[B_{n}]/Z(A[B_{n}])$ is not surjective. Indeed, recall that $A[B_{n}]=A[\tilde A_{n-1}]\rtimes \langle \rho_B\rangle$, with the center $Z(A[B_{n}])$ generated by $\Delta_B=\rho_B^{n}$, and thus $Z(A[B_{n}])\cap A[\tilde A_{n-1}]=1$. Hence
$A[B_{n}]/Z(A[B_{n}])=A[\tilde A_{n-1}]\rtimes \langle \bar\rho_B\rangle$, where $\bar\rho_B=\pi(\rho_B)$. We observe that in case (3) the image $\Im(\pi\circ\varphi)$ lies entirely in the $A[\tilde A_{n-1}]$ factor of $A[\tilde A_{n-1}]\rtimes \langle \bar\rho_B\rangle$, and so it does not contain $\bar\rho_B$.

To analyze endomorphisms $\varphi$ of types (2a) and (2b), we observe that they are transvections of certain automorphisms, i.e.\ such an endomorphism has form: $\varphi(x)=\phi(x)\cdot\lambda(x)$ for any $x\in A[B_{n}]$, where $\phi$ is an automorphism of $A[B_{n}]$ and $\lambda\colon A[B_{n}]\to Z(A[B_{n}])$ is a homomorphism, defined as follows:
\[
\text{(2a):\quad}
\begin{aligned}
\phi(t_i)&=t_i^\varepsilon\,,\\
\phi(\rho_B)&=\rho_B\,,
\end{aligned}
\qquad\qquad\text{(2b):\quad}
\begin{aligned}
\phi(t_i)&=t_{n-i}^\varepsilon\,,\\
\phi(\rho_B)&=\rho_B^{-1}\,,
\end{aligned}
\qquad\qquad\qquad
\begin{aligned}
\lambda(t_i)&=\Delta_B^p\,,\\
\lambda(\rho_B)&=\Delta_B^q\,.
\end{aligned}
\]
The statements of Corollary~\ref{cor2_4} will be deduced from the following general lemma.
\begin{lem}\label{lem_transvections}
Let $G$ be a group with center $Z$, and let an endomorphism $\varphi$ of $G$ be a transvection of an automorphism, i.e.\ $\varphi=\phi\cdot\lambda$ with $\phi$ an automorphism of $G$ and $\lambda\colon G\to Z$ some homomorphism. Then $\varphi(Z)\subseteq Z$ and for the restriction $\varphi|_Z\colon Z\to Z$ we have:
\begin{enumerate}
\item $\varphi$ is injective if and only if $\varphi|_Z$ is injective;
\item $\varphi$ is surjective if and only if $\varphi|_Z$ is surjective (i.e.\ $\Im(\varphi|_Z)=Z$).
\end{enumerate}
\end{lem}
\begin{proof}
(1) The kernel of $\varphi$ is equal to $\ker\varphi=\{x\in G\mid \phi(x)\cdot\lambda(x)=1\}$. Thus, if $x\in\ker\varphi$, then $\phi(x)=\lambda(x)^{-1}\in Z$, hence $x\in Z$. Thus $\ker\varphi\subseteq Z$, which proves that if $\varphi|_Z$ is injective then $\varphi$ is injective. The opposite statement is obvious.

(2) Assume that $\varphi$ is surjective. Then for each $x\in G$ there exists $y\in G$ such that $x=\varphi(y)=\phi(y)\lambda(y)$. In particular, taking $x\in Z$, we get $\phi(y)=x\lambda(y)^{-1}\in Z$, hence $y\in Z$. I.e.\ $x=\varphi(y)$ for $y\in Z$, which means that $Z\subseteq \Im(\varphi|_Z)$, i.e.\ that $\varphi|_Z$ is surjective.

Now let $\varphi|_Z$ be surjective. Take arbitrary $x\in G$ and let $z\in Z$ be such that $\varphi(z)=\lambda(x)$. Then $\varphi(x)=\phi(x)\lambda(x)=\phi(x)\varphi(z)$ so that $\phi(x)=\varphi(x)\varphi(z)^{-1}\in\Im\varphi$. Since $\phi$ is an automorphism and $x$ is arbitrary, we conclude that $G=\Im\phi\subseteq \Im\varphi$, i.e.\ that $\varphi$ is surjective.
\end{proof}
To finish the proof of Corollary~\ref{cor2_4} we observe that $\varphi$ acts on the generator $\rho_B^{n}$ of the cyclic center of $A[B_{n}]$ as follows:
\[
\begin{aligned}
\varphi(\rho_B^{n})&=(\rho_B\Delta_B^q)^{n}=\Delta_B^{1+qn}\text{\quad\qquad in case (2a)},\\
\varphi(\rho_B^{n})&=(\rho_B^{-1}\Delta_B^q)^{n}=\Delta_B^{-1+qn}\text{\qquad in case (2b)}.
\end{aligned}
\]
In particular, we see from Lemma~\ref{lem_transvections} that $\varphi$ is always injective, and that $\varphi$ is surjective if and only if $\Delta_B^{qn\pm1}\in\{\Delta_B,\Delta_B^{-1}\}$, which is true if and only if $q=0$, since $n>2$.
\end{proof}

\begin{proof}[Proof of Corollary~\ref{cor2_5}]
We see from Corollary~\ref{cor2_4} that an arbitrary automorphism $\varphi$ of $A[B_{n}]$ is conjugate to one of the following two types:
\[
\text{(2a):\qquad}
\begin{aligned}
&\varphi(t_i)=t_i^\varepsilon\Delta_B^{p}\\
&\varphi(\rho_B)=\rho_B
\end{aligned}
\qquad\qquad \text{or} \qquad\qquad \text{(2b):\qquad}
\begin{aligned}
&\varphi(t_i)=t_{n-i}^\varepsilon\Delta_B^{p}\\
&\varphi(\rho_B)=\rho_B^{-1}
\end{aligned}
\]
for $0\le i\le n-1$, $p\in \Z$ and $\varepsilon=\pm1$. 
We see that an automorphism of type (2a) is equal to $T^p$ if $\varepsilon=1$ and to $\mu T^p$ if $\varepsilon=-1$, and an automorphism of type (2b) is equal to $\tau T^{-p}$ if $\varepsilon = -1$ and to $\mu \tau T^{-p}$ if $\varepsilon=1$.
Hence $\Aut(A[B_{n}])$ is generated by $\Inn(A[B_{n}])$, $T$, $\tau$ and $\mu$. 
We check directly that $\mu^2 = \tau^2 =1$, $\mu \tau = \tau \mu$, $\mu T \mu = T^{-1}$, and $\tau T = T \tau$.
Denote by $\xi\colon A[B_{n}]\to\Z$ the homomorphism that sends $t_i\mapsto 1$ and $\rho_B\mapsto1$. 
We observe that inner automorphisms preserve $\xi$, i.e.\ $\xi(gxg^{-1})=\xi(x)$ for any $g,x\in A[B_{n}]$. 
On the other hand, an arbitrary automorphism $\varphi$ of type (2b) above is not inner, since $\xi(\varphi(\rho_B))=-\xi(\rho_B)$, and $\xi(\rho_B)\ne 0$. 
For an automorphism $\varphi$ of type (2a) we have $\xi (\varphi(t_i))=\xi(t_i^\varepsilon\rho_B^{np})=\pm1+np\ne1=\xi(t_i)$, as $n\ge5$, 
except if $\varepsilon=1$ and $p=0$, which means $\varphi$ is the identity.
So, $\varphi$ is not inner in case (2a) if different from $\id$. 
We conclude that $\Inn(A[B_{n}])\cap\langle T,\tau,\mu\rangle=\{\id\}$.

Clearly, $\tau\notin\langle T,\mu\rangle$, since $T$ and $\mu$ fix $\rho_B$, but $\tau$ does not. Since $\tau$ and $\mu$ are commuting involutions and $T$ has infinite order, $T\notin\langle\tau,\mu\rangle$. Also $\mu(\rho_B)=\rho_B\ne\rho_B^{-1}=\tau(T^k(\rho_B))$ for any $k\in\Z$, hence $\mu\ne \tau T^k$ for any $k\in\Z$, i.e.\ $\mu\notin \langle T,\tau\rangle$ either. Thus all the decompositions of Corollary~\ref{cor2_5} are true as stated.
\end{proof}

\section{Proofs for endomorphisms of \texorpdfstring{$A[B_n]/Z(A[B_n])$}{A[Bn]/Z(A[Bn])}}\label{sec5}

\begin{proof}[Proof of Proposition~\ref{prop_lifts}]
Let $\phi\colon \ov{A[B_{n}]}\to \ov{A[B_{n}]}$ be an endomorphism. We are going to construct an endomorphism $\varphi\colon A[B_{n}]\to A[B_{n}]$ such that $\ov{\varphi}=\phi$. For that, we define elements $u_i\in A[B_{n}]$ which serve as the images $\varphi(r_i)$ of the standard generators $r_i$, in such a way that all Artin relations between $u_i$ are satisfied. Recall that the center of $A[B_{n}]$ is the cyclic group generated by $\Delta_B$. 
First, we choose any $u_1\in A[B_{n}]$ such that $\pi(u_1)=\phi(\bar r_1)$. Assume by induction that $2\le i\le n-1$ and an element $u_{i-1}$ with the property $\pi(u_{i-1})=\phi(\bar r_{i-1})$ is already defined. Choose arbitrary $u_i'\in A[B_{n}]$ such that $\pi(u_i')=\phi(\bar r_i)$. Since $\phi(\bar r_{i-1}\,\bar r_i\,\bar r_{i-1})=\phi(\bar r_i\,\bar r_{i-1}\,\bar r_i)$, there exists $k_i\in\Z$ such that $u_{i-1}\,u_i'\,u_{i-1}=u_i'\,u_{i-1}\,u_i'\,\Delta_B^{k_i}$. If we now define $u_i=u_i'\,\Delta_B^{k_i}$, we have $\pi(u_i)=\pi(u_i')=\phi(\bar r_i)$ and
\[
u_{i-1}\,u_i\,u_{i-1}=u_{i-1}\,u_i'\,u_{i-1}\,\Delta_B^{k_i}=u_i'\,u_{i-1}\,u_i'\,\Delta_B^{2k_i}=(u_i'\Delta_B^{k_i})\,u_{i-1}\,(u_i'\,\Delta_B^{k_i})=u_i\,u_{i-1}\,u_i\,,
\]
so that the braid relations between generators $u_i$ and $u_{i-1}$ are satisfied. 

Choose additionally $u_{n}\in A[B_{n}]$ to be an arbitrary element such that $\pi(u_{n})=\phi(\bar r_{n})$.

We claim that all other Artin relations between elements $u_i$ are also satisfied. To prove that, consider the homomorphism $z\colon A[B_{n}]\to\Z$, $r_i\mapsto 1$ for all $1\le i\le n$. If $i,j\in\{1,\dots,n\}$ are such that $r_ir_j=r_jr_i$, then we have $\phi(\bar r_i\,\bar r_{j})=\phi(\bar r_{j}\,\bar r_i)$, and hence $u_iu_j=u_ju_i\,\Delta_B^k$ for some $k\in\Z$. Applying the homomorphism $z$ to both sides of the latter equality, and recalling that $\Delta_B=(r_1r_2\dots r_n)^n$, we see that $z(\Delta_B)=n^2$ and
\[
z(u_i)+z(u_j)=z(u_j)+z(u_i)+kn^2,
\]
hence $k=0$, and therefore $u_iu_j=u_ju_i$.

Similarly, for elements $r_{n-1}$ and $r_{n}$ we have: $\phi(\bar r_{n-1}\,\bar r_{n}\,\bar r_{n-1}\,\bar r_{n})=\phi(\bar r_{n}\,\bar r_{n-1}\,\bar r_{n}\,\bar r_{n-1})$ so that
\[
u_{n-1}\,u_{n}\,u_{n-1}\,u_{n}=u_{n}\,u_{n-1}\,u_{n}\,u_{n-1}\,\Delta_B^\ell
\]
for some $\ell\in\Z$. Applying $z$ to both sides, we get:
\[
2z(u_{n-1})+2z(u_{n})=2z(u_{n})+2z(u_{n-1})+\ell n^2,
\]
and we again conclude that $\ell=0$ and that $u_{n-1}\,u_{n}\,u_{n-1}\,u_{n}=u_{n}\,u_{n-1}\,u_{n}\,u_{n-1}$.

Hence the correspondence $r_i\mapsto u_i$ for $1\le i\le n$, defines a homomorphism $\varphi\colon A[B_{n}]\to A[B_{n}]$, which is a lift of $\phi$.
\end{proof}

\begin{proof}[Proof of Theorem~\ref{thm2_8}]
Let $n\ge5$ and $\phi\colon \ov{A[B_{n}]}\to \ov{A[B_{n}]}$ be an endomorphism. We know from Proposition~\ref{prop_lifts} that $\phi$ admits a lift $\varphi\colon A[B_{n}]\to A[B_{n}]$. By Theorem~\ref{thm2_1}, there are three possibilities for $\varphi$ up to conjugation. We will analyze them and deduce the structure of $\phi$ in each of these cases.

{\it Case (1):} there exist $g,h\in A[B_{n}]$ such that $gh=hg$, and for all $0\le i\le n-1$ we have $\varphi(t_i)=g$ and $\varphi(\rho_B)=h$. Then, since $\varphi(Z(A[B_{n}]))\subseteq Z(A[B_{n}])$, we have $\varphi(\rho_B^{n})=\rho_B^{kn}$ for some $k\in\Z$. Hence $h^{n}=\varphi(\rho_B^n)=\rho_B^{kn}$. By \cite[Theorem~1.1]{LeeLee1}, it follows that $h$ and $\rho_B^k$ are conjugate in $A[B_{n}]$, which means that we may assume without loss of generality that $\varphi(\rho_B)=\rho_B^k$ and $\varphi(t_i)$ belongs to the centralizer $Z_{A[B_{n}]}(\rho_B^k)$ of $\rho_B^k$ in $A[B_{n}]$. On the other hand, a description of $Z_{A[B_{n}]}(\rho_B^k)$ is known from \cite[Proposition~3.5]{GWie1}, and it can be adapted for our situation as follows.

\begin{lem}[Gonz\'alez-Meneses--Wiest~\cite{GWie1}]
Let $d=\gcd(n,k)$ and $r={n}/{d}$. If $d=1$, then $Z_{A[B_{n}]}(\rho_B^k)=\langle \rho_B\rangle$, and if $d\ge 2$ then $Z_{A[B_{n}]}(\rho_B^k)=\langle \rho_B, t_{d}t_{2d}\dots t_{rd}\rangle$, with the convention that $t_{rd}=t_0$. In particular, if $d=n$, i.e.
 if $k$ is a multiple of $n$, then $Z_{A[B_{n}]}(\rho_B^k)=\langle \rho_B, t_0\rangle=A[B_{n}]$.\qed
\end{lem}

Projecting to $\ov{A[B_{n}]}$, we get the following description of $\phi$ in Case (1):	
\begin{align*}
\phi(\bar\rho_B)&=\bar \rho_B^\kappa,\quad\text{for some}\quad\kappa\in\{0,\dots,n-1\},\\
\phi(\bar t_i)&=\bar g,
\end{align*}
with $\bar g$ belonging to $\pi(Z_{A[B_{n}]}(\rho_B^\kappa))$, which is equal to $\langle\bar \rho_B\rangle$ if $\gcd(n,\kappa)=1$ and to $\langle\bar\rho_B,\bar t_{d}\bar t_{2d}\dots\bar t_0\rangle$ if $d=\gcd(n,\kappa)\ne1$. Notice that, in general, we have $\pi(Z_{A[B_{n}]}(x))=Z_{\ov{A[B_{n}]}}(\pi(x))$ for arbitrary $x\in A[B_{n}]$. (Indeed, denoting $Z=Z(A[B_{n}])$, we have $Z_{\ov{A[B_{n}]}}(\pi(x))=\{\pi(g)\in \ov{A[B_{n}]} \mid gxg^{-1}\in xZ\}$. Now, using the homomorphism $\xi\colon A[B_{n}]\to \Z$, $t_i\mapsto 1$, $\rho_B\mapsto 1$, we observe that if $gxg^{-1}\in xZ$ then $gxg^{-1}=x$, for arbitrary $g\in A[B_{n}]$, i.e.\ $g\in Z_{A[B_n]}(x)$.) Thus we can write that $\bar g\in Z_{\ov{A[B_{n}]}}(\bar\rho_B^\kappa)$ in this case.

{\it Case (2):} there exist $\varepsilon\in\{\pm1\}$ and $p,q\in\Z$ such that for all $0\le i\le n-1$ we have one of the two options:
\[
\text{(2a):\qquad}
\begin{aligned}
\varphi(t_i)&=t_i^\varepsilon\,\Delta_B^{p}\\
\varphi(\rho_B)&=\rho_B\,\Delta_B^{q}
\end{aligned}
\qquad\qquad \text{or} \qquad\qquad \text{(2b):\qquad}
\begin{aligned}
\varphi(t_i)&=t_{n-i}^\varepsilon\,\Delta_B^{p}\\
\varphi(\rho_B)&=\rho_B^{-1}\,\Delta_B^{q}.
\end{aligned}
\]
We see that in both these cases $\varphi(\rho_B^{n})=\rho_B^{\pm n}\Delta_B^{qn}=\Delta_B^{qn\pm1}\in Z(A[B_{n}])$, as needed, and for $\phi=\ov{\varphi}$ we have:
\[
\text{(2a):\qquad}
\begin{aligned}
\phi(\bar t_i)&=\bar t_i^\varepsilon\,,\\
\phi(\bar\rho_B)&=\bar\rho_B\,,
\end{aligned}
\qquad\qquad \text{or} \qquad\qquad \text{(2b):\qquad}
\begin{aligned}
\phi(\bar t_i)&=\bar t_{n-i}^\varepsilon\,,\\
\phi(\bar \rho_B)&=\bar \rho_B^{-1}\,.
\end{aligned}
\]

{\it Case (3):} In this case the endomorphism $\varphi$ has the property $\varphi(\rho_B)=\rho_k\Delta_Y^{2r}\Delta_B^s$ for $k\in\{0,1\}$, $r,s\in\Z$. In particular, $\rho_k\Delta_Y^{2r}\in A[\tilde A_{n-1}]$ and $A[\tilde A_{n-1}]$ intersects trivially with $\ker\pi=\langle \rho_B^{n}\rangle$. Hence, $\pi(\varphi(\rho_B^{n}))=\pi(\rho_k^{n}\Delta_Y^{2rn})\ne1$, which means that $\varphi(\rho_B^{n})$ does not belong to $Z(A[B_{n}])$. We conclude that in Case (3), $\varphi(Z(A[B_n]))\not\subseteq Z(A[B_n])$, hence $\varphi$ cannot be a lift of an endomorphism of $\ov{A[B_{n}]}$.
\end{proof}

\begin{proof}[Proof of Proposition~\ref{prop2_9}]
If $\phi$ belongs to case (1), then the image of $\phi$ is abelian, and hence $\phi$ is non-injective, whereas if $\phi$ belongs to case (2a) or (2b), then $\phi$ is actually an automorphism. If $\phi=\conj_x\circ\psi$, then $\psi$ has the same properties, since $\conj_x$ is an automorphism. Thus, $\phi$ and $\psi$ cannot belong to case (1) and to either case (2a) or (2b) simultaneously. 
Consider a homomorphism 
\[
\bar\eta\colon \ov{A[B_n]}\longrightarrow \Z\oplus\Z/n\Z, \quad \bar t_i\longmapsto (1,0)\text{\quad for all $i$}, \quad\bar\rho_B\longmapsto (0,1).
\]
Notice that the image of $\bar\eta$ is an abelian group, hence $\bar\eta\circ\conj_x\circ\phi=\bar\eta\circ\phi$, for any $\phi$ and $x$, so it will suffice to show that cases (2a) and (2b) are distinguished by the image of the composition $\bar\eta\circ\phi$ and that parameters $\varepsilon$ and $\kappa$ are determined by $\bar\eta\circ\phi$ uniquely. We have the following formulas for $\bar\eta\circ\phi$:
\[
\begin{aligned}
\text{case (1):\qquad}\bar\eta\circ\phi:&\quad \bar t_i\longmapsto(\xi(g),0),\quad \bar\rho_B\longmapsto(0,\kappa);\\
\text{case (2a):\qquad}\bar\eta\circ\phi:&\quad \bar t_i\longmapsto(\varepsilon,0),\quad \bar\rho_B\longmapsto(0,1);\\
\text{case (2b):\qquad}\bar\eta\circ\phi:&\quad \bar t_i\longmapsto(\varepsilon,0),\quad \bar\rho_B\longmapsto(0,n-1).\\
\end{aligned}
\]
where $\xi$ denotes the homomorphism $\xi\colon A[\tilde A_{n-1}]\to\Z$, sending $\bar t_i\mapsto 1$ for all $0\le i\le n-1$. From these formulas we see that endomorphism $\phi$ cannot belong to cases (2a) and (2b) simultaneously, since $\bar\eta(\phi(\bar\rho_B))$ is different in the two cases. And clearly, parameters $\varepsilon$ and $\kappa$ are determined uniquely by these formulas.
\end{proof}

\begin{proof}[Proof of Corollary~\ref{cor2_10}]
From Theorem~\ref{thm2_8} we see that endomorphisms of case (1) are not surjective, since their image is abelian, and hence all automorphisms of $\ov{A[B_{n}]}$ up to conjugation come from case~(2). Hence $\Aut(\ov{A[B_{n}]})=\langle\Inn(\ov{A[B_{n}]}),\bar\tau,\bar\mu\rangle$. Since, in general, for an arbitrary automorphism $\alpha\in\Aut(G)$ and arbitrary element $g\in G$ of a group $G$ we have $\alpha\circ\conj_g\circ\alpha^{-1}=\conj_{\alpha(g)}$, the conjugation by $\bar\tau$ and by $\bar\mu$ inside $\Aut(\ov{A[B_{n}]})$ leaves $\Inn(\ov{A[B_{n}]})$ invariant. Also notice that $\bar\tau$ and $\bar\mu$ are commuting involutions. So to establish that $\Aut(\ov{A[B_{n}]})=\Inn(\ov{A[B_{n}]})\rtimes \langle\bar\tau,\bar\mu\rangle$ we just need to prove that $\langle \bar\tau, \bar\mu\rangle\cap\Inn(\ov{A[B_{n}]})=\{\id\}$.

Consider the homomorphism $\bar\xi\colon \ov{A[B_{n}]}\to \Z/n\Z$, $\bar t_i\mapsto 1$, $\bar\rho_B\mapsto 1$. Since $Z(A[B_{n}])=\langle \rho_B^{n}\rangle$, $\bar\xi$ is well-defined. We notice that for any $\bar g\in \ov{A[B_{n}]}$, we have $\bar\xi(\bar g x\bar g^{-1})=\bar\xi(x)$ for all $x\in \ov{A[B_{n}]}$, i.e.\ $\bar\xi$ is invariant under conjugation. Now we observe that $\bar\xi(\bar\tau\bar\mu(\bar\rho_B))=\bar\xi(\bar\tau(\bar\rho_B))=\bar\xi(\bar\rho_B^{-1})=-1$ whereas $\bar\xi(\bar\rho_B)=1$, which proves that automorphisms $\bar\tau$ and $\bar\tau\bar\mu$ are not inner. Similarly $\bar\xi(\bar\mu(\bar t_i))=\bar\xi(\bar t_i^{-1})=-1\ne1=\bar\xi(\bar t_i)$, so that $\bar\mu$ is not an inner automorphism either, which proves that $\langle \bar\tau, \bar\mu\rangle\cap\Inn(\ov{A[B_{n}]})=\{\id\}$.

To conclude that $\Inn(\ov{A[B_{n}]})\simeq \ov{A[B_{n}]}$ we prove that the center of $\ov{A[B_{n}]}$ is trivial (this fact is known to experts, though is not readily available in the literature). Indeed, let $\bar z\in Z(\ov{A[B_{n}]})$ and let $z\in A[B_{n}]$ be such that $\pi(z)=\bar z$. Then for an arbitrary $x\in A[B_{n}]$ we have $zxz^{-1}=xz'$ for some $z'\in Z(A[B_{n}])=\langle\rho_B^{n}\rangle$. Considering the homomorphism $\xi\colon A[B_{n}]\to\Z$, $t_i\mapsto 1$, $\rho_B\mapsto 1$, we notice that $\xi(x)=\xi(zxz^{-1})=\xi(xz')=\xi(x)+\xi(z')$, so that $\xi(z')=0$. But $z'=\rho_B^{kn}$ for some $k\in\Z$, hence, $\xi(z')=kn$ which is equal to $0$ if and only if $k=0$, i.e.\ if and only if $z'=1$. Hence $zxz^{-1}=x$ for every $x\in A[B_{n}]$, which means that $z\in Z(A[B_{n}])$ and $\bar z=1$.
\end{proof}

\begin{proof}[Proof of Corollary~\ref{cor2_12}]
Indeed, recall that $\ov{A[B_{n}]}=A[\tilde A_{n-1}]\rtimes\langle \bar\rho_B\rangle$, and from Corollary~\ref{cor2_10} we see that $\Aut(\ov{A[B_{n}]})\simeq \Inn(\ov{A[B_{n}]})\rtimes \langle\bar\tau,\bar\mu\rangle\simeq \bigl(A[\tilde A_{n-1}]\rtimes\langle \bar\rho_B\rangle\bigr) \rtimes \langle\bar\tau,\bar\mu\rangle$. Clearly, conjugation by $\bar\rho_B$ preserves $A[\tilde A_{n-1}]$ and we see directly that $\bar\tau$, $\bar\mu$ leave $A[\tilde A_{n-1}]$ invariant. This shows that $A[\tilde A_{n-1}]$ is a characteristic subgroup of $\ov{A[B_n]}$. On the other hand, there exist endomorphisms of case~(1) of Theorem~\ref{thm2_8} which send $\bar t_i$ to a nontrivial element of $\langle\bar\rho_B\rangle$, i.e.\ $A[\tilde A_{n-1}]$ is not a fully characteristic subgroup of $\ov{A[B_{n}]}$.
\end{proof}

\section{Additional remarks}\label{sec_rem}
\begin{rem}\label{rem:delta}
The word representing the element $\delta$, which participates in the expressions for endo- and automorphisms of $A[B_{n}]$ in terms of the standard generators of $A[B_{n}]$,
\[
\delta=r_{n-1}r_{n-2}\dots r_2\,(r_1^2)\,r_2\dots r_{n-2}r_{n-1}\,,
\]
appears in a few other contexts in the literature. One situation where this expression shows up is a natural embedding of the Artin group $A[A_{n-2}]$ into that of type $A[A_{n-1}]$. 
To keep notation the same, assume that $A[A_{n-1}]=A_{X_1}[B_n]$ is generated by $X_1=\{r_1,\dots,r_{n-1}\}$ and that $A[A_{n-2}]=A_{X_2}[B_n]$ is generated by $X_2=\{r_1,\dots,r_{n-2}\}$. 
Then the element $\delta$ is equal to the quotient of two generators of the centers of $A[A_{n-1}]$ and $A[A_{n-2}]$, i.e.\ 
\[
\delta = \Delta_{X_1}[B_n]^2\cdot\Delta_{X_2}[B_n]^{-2}\,, 
\]
where $\Delta_{X_1}[B_n]$ (resp. $\Delta_{X_2}[B_n]$) denotes the Garside element of $A[A_{n-1}]=A_{X_1}[B_n]$ (resp. $A[A_{n-2}]=A_{X_2}[B_n]$).
The proof of this follows from~\cite[Lemma~5.1]{CasPar1}.

Another context where the element $\delta$ appears, is the presentation of the mapping class group of a closed orientable surface. Namely, let $\Sigma_g$ be a closed orientable surface of genus $g$, and let $c_1,\dots,c_{2g+1}$ be a chain of isotopy classes of simple closed curves in $\Sigma_g$, for which $i(c_k,c_{k+1})=1$ and $i(c_k,c_l)=0$ for $|k-l|>1$. Then the so-called hyperelliptic involution $h$ of $\Sigma_g$ is given by the formula (see~\cite[section~5.1.4]{FarMar1}):
\[
h=T_{c_{2g+1}}T_{c_{2g}}\dots T_{c_1}T_{c_1}\dots T_{c_{2g}}T_{c_{2g+1}},
\]
which formally coincides with the expression for $\delta$ if one sets $r_i=T_{c_i}$, for all $1\le i\le n-1=2g+1$. This coincidence can be explained as follows. It is well-known that the so-called chain relations between the Dehn twists corresponding to a chain of simple closed curves on a surface can be written in terms of the generators of the centers of the corresponding Artin subgroups. Namely (see~\cite[Proposition~2.12]{LabPar1}),
\[
\Delta(T_{c_1},\dots,T_{c_{2g+1}})^2=(T_{c_1}\dots T_{c_{2g+1}})^{2g+2}=T_{d_1}T_{d_2},
\]
where $d_1$, $d_2$ are the two boundary components of the regular neighborhood of the chain $\{c_i\}_{i=1}^{2g+1}$. For an even number of curves we get:
\[
\Delta(T_{c_1},\dots,T_{c_{2g}})^4=\bigl((T_{c_1}\dots T_{c_{2g}})^{2g+1}\bigr)^2=T_d,
\]
where $d$ is the single boundary component of the regular neighborhood of the chain $\{c_i\}_{i=1}^{2g}$. It is also known (see e.g.~\cite[Table~1]{Matsu1}) that the square of $\Delta(T_{c_1},\dots,T_{c_{2g}})$ is equal to the hyperelliptic involution $h$ on $\Sigma_g$. Since on the closed surface $\Sigma_g$ all twists $T_{d_1}$, $T_{d_2}$, $T_d$ are trivial, we get:
\[
h=\Delta(T_{c_1},\dots,T_{c_{2g}})^2=\id\cdot\Delta(T_{c_1},\dots,T_{c_{2g}})^{-2}=\Delta(T_{c_1},\dots,T_{c_{2g+1}})^2\cdot \Delta(T_{c_1},\dots,T_{c_{2g}})^{-2},
\]
which is exactly the above expression for $\delta$ as a quotient of the two generators of the centers of the corresponding Artin subgroups.
\end{rem}

\begin{rem}
There is a considerable interest in finding embeddings of different Artin groups into one another and into mapping class groups, see e.g.~\cite{Crisp2,BelMar2,Morta1,Ryffe1,Pool1}. In this regard, Corollary~\ref{cor2_3} provides a few new ways to embed Artin group $A[I_2(4)]=A[B_2]$ into $A[B_{n}]$ and $A[A_{n}]$. Namely, since $r_{n-1}$ and $r_{n}$ satisfy the Artin relation of length four: $r_{n-1}\,r_{n}\,r_{n-1}\,r_{n}=r_{n}\,r_{n-1}\,r_{n}\,r_{n-1}$, the elements $\varphi(r_{n-1})$ and $\varphi(r_{n})$ also satisfy this relation, for any homomorphism $\varphi$. In their simplest form, for endomorphisms $\varphi$ of types (IIIb) and (IIb) of Corollary~\ref{cor2_3}, these relations read:
\begin{align}
\delta\, r_{n-1}^{-1}\,\delta \,r_{n-1}^{-1} &= r_{n-1}^{-1}\,\delta\, r_{n-1}^{-1}\,\delta\,,\label{art:1}\\
(\delta r_{n})\, r_{n-1}^{-1}\,(\delta r_{n})\, r_{n-1}^{-1} &= r_{n-1}^{-1}\,(\delta r_{n})\,r_{n-1}^{-1}\,(\delta r_{n})\,.\label{art:2}
\end{align}
Denote $\delta'=r_{n-1}^{-1}\,\delta\, r_{n-1}^{-1}=r_{n-2}\dots r_2\,(r_1)^2\,r_2\dots r_{n-2}$. Then relation~\eqref{art:1} can be written as $\delta\,\delta'=\delta'\,\delta$ and it is equivalent to yet another Artin relation of length four:
\begin{equation}
\delta'\, r_{n-1}\, \delta'\, r_{n-1} = r_{n-1}\, \delta'\, r_{n-1}\, \delta'\,.\label{art:3}
\end{equation}
Of course, these relations are satisfied since $\varphi$ is a homomorphism, by Corollary~\ref{cor2_3}. However, one can see the validity of these relations directly. For relations~\eqref{art:1} and \eqref{art:3}, we recall the expression for $\delta$ and $\delta'$ as the quotient of two generators of centers of the respective parabolic subgroups (see Remark~\ref{rem:delta}): 
$\delta=\Delta_{X_1}[B_n]^2\cdot\Delta_{X_2}[B_n]^{-2}$ and 
$\delta'=\Delta_{X_2}[B_n]^2\cdot\Delta_{X_3}[B_n]^{-2}$, where $X_3 = \{r_1, \dots, r_{n-3}\}$.
Since all the four participating factors commute with one another, the identity $\delta\,\delta'=\delta'\,\delta$ holds.

To see the validity of relation~\eqref{art:2} directly, we note that $r_{n}$ commutes with $\delta'$ and that $r_{n-1}^{-1}\,\delta=\delta'\,r_{n-1}$. Then (we apply obvious identities to the underlined subwords):
\begin{multline*}
(\delta r_{n})\, \underline{r_{n-1}^{-1}\,(\delta}\, r_{n})\, r_{n-1}^{-1} = \delta r_{n}\, \delta'\,\underline{r_{n-1} r_{n}r_{n-1}^{-1}r_{n}^{-1}}\,\cdot r_{n}= \delta\, \underline{r_{n}\, \delta'\,r_{n}^{-1}}\, r_{n-1}^{-1}r_{n} r_{n-1}\,\cdot r_{n}=\\
\underline{\delta\,\delta'}\,r_{n-1}^{-1}\,r_{n} r_{n-1}\,r_{n}=
\underline{\delta'\,\delta\,r_{n-1}^{-1}}\,r_{n} r_{n-1}\,r_{n}=
r_{n-1}^{-1}\delta\,\underline{r_{n-1}^{-1}\delta\,r_{n-1}^{-1}}\,r_{n}\, r_{n-1}\,r_{n}=
r_{n-1}^{-1}\delta\,\underline{\delta'\,r_{n}}\, r_{n-1}\,r_{n}=\\
r_{n-1}^{-1}\,(\delta\,r_{n})\,\underline{\delta'\,r_{n-1}}\, r_{n}=
r_{n-1}^{-1}\,(\delta\,r_{n})\,r_{n-1}^{-1}\,(\delta\, r_{n})\,.
\end{multline*}
\end{rem}

\section{Some open questions}\label{sec_oq}

In recent years there has been a series of publications concerning automorphisms and endomorphisms of certain types of spherical and affine Artin groups. As a rule, in these publications small rank cases were excluded, since the methods employed were not applicable for them. In this section we summarized what is known and what is not known concerning automorphisms and endomorphisms of Artin groups of these types, to the best of our knowledge. Ideally, the problem of determining the automorphism group of a group $G$ should involve obtaining a description of $\Aut(G)$, $\Out(G)$ and deciding whether the short exact sequence $1\to\Inn(G)\to\Aut(G)\to\Out(G)\to1$ splits. The problem of determining the endomorphisms of a group $G$ involves obtaining a sensible explicit description of all endomorphisms of $G$, possibly, up to an equivalence of some sort.

\subsubsection*{Type \texorpdfstring{$I_2(m)$, $m\ge3,\, m\ne\infty$}{I2(m)}:}
\begin{itemize}
\item Coxeter graph:\qquad
\begin{tikzpicture}[scale=0.7,double distance=2.3pt,thick]
\begin{scope}[scale=1.33] 
\fill (0,0) circle (2.3pt); 
\fill (1,0) circle (2.3pt); 
\draw (0,0)--(1,0);
\draw (0.5,0.25) node {\small $m$};
\end{scope}
\end{tikzpicture}
\item Automorphisms: $\Out(A[I_2(m)])$ was determined in~\cite[Theorem~E]{GiHoMeRa1}, $\Aut(A[I_2(m)])$ was determined in~\cite[Theorem~5.1]{CriPar1}.
\item Endomorphisms: the problem is open (for $m=3$ see type $A_n$ for $n=2$ below).
\end{itemize}

\subsubsection*{Type \texorpdfstring{$A_n$, $n\ge1$}{An}:}
\begin{itemize}
\item Coxeter graph:\qquad \raisebox{-2.5ex}{
\begin{tikzpicture}[scale=0.7,double distance=2.3pt,thick]
\begin{scope}[scale=1.33] 
\fill (0,0) circle (2.3pt) node [below=2pt,blue] {\scriptsize$1$};
\fill (1,0) circle (2.3pt) node [below=2pt,blue] {\scriptsize$2$};
\fill (2,0) circle (2.3pt) node [below=2pt,blue] {\scriptsize$3$};
\fill (4,0) circle (2.3pt) node [below=1pt,blue] {\scriptsize$n{-}1$};
\fill (5,0) circle (2.3pt) node [below=2pt,blue] {\scriptsize$n$};
\draw [very thick] (0,0)--(1,0)--(2,0)--(2.5,0);
\draw [very thick,dashed] (2.5,0)--(3.5,0);
\draw [very thick] (3.5,0)--(4,0)--(5,0);
\end{scope}
\end{tikzpicture}
}
\item Automorphisms: determined in~\cite{DyeGro1} and, independently, in~\cite{ChaCri1} for $n\ge3$. For $n=2$ see also the references for type $I_2(m)$ with $m=3$. 
\item Endomorphisms: determined for $n\ge5$ in \cite{Caste1}, and, independently, for $n\ge4$ in~\cite{ChKoMa1}, for $n\ge3$ in~\cite{Orevk1}, and for $n=2$ a description without proof is given in~\cite[Proposition~1.8]{Orevk1}. 
\end{itemize}

\subsubsection*{Type \texorpdfstring{$B_n$, $n\ge2$}{Bn}:}
\begin{itemize}
\item Coxeter graph:\qquad \raisebox{-2.5ex}{
\begin{tikzpicture}[scale=0.7]
\begin{scope}[scale=1.33] 
\fill (0,0) circle (2.3pt) node [below=2pt,blue] {\scriptsize$1$};
\fill (1,0) circle (2.3pt) node [below=2pt,blue] {\scriptsize$2$};
\fill (2,0) circle (2.3pt) node [below=2pt,blue] {\scriptsize$3$};
\fill (4,0) circle (2.3pt) node [below=1pt,blue] {\scriptsize $n{-}1$};
\fill (5,0) circle (2.3pt) node [below=2pt,blue] {\scriptsize $n$};
\draw [very thick] (0,0)--(1,0)--(2,0)--(2.5,0);
\draw [very thick,dashed] (2.5,0)--(3.5,0);
\draw [very thick] (3.5,0)--(4,0)--(5,0);
\draw (4.5,0.3) node {\small $4$};
\end{scope}
\end{tikzpicture}
}
\item Automorphisms: determined in~\cite{ChaCri1} for $n\ge3$ and, independently, in the present paper for $n\ge5$. For case $n=2$, see type $I_2(m)$ with $m=4$.
\item Endomorphisms: determined for $n\ge5$ in the present paper. The problem is open for $n=2,3,4$.
\end{itemize}

\subsubsection*{Type \texorpdfstring{$D_n$, $n\ge4$}{Dn}:}
\begin{itemize}
\item Coxeter graph:\qquad \raisebox{-4.5ex}{
\begin{tikzpicture}[scale=0.7, very thick]
\begin{scope}[scale=1.33] 
\fill (0,0) circle (2.3pt) node [below=2pt,blue] {\scriptsize$1$}; 
\fill (1,0) circle (2.3pt) node [below=2pt,blue] {\scriptsize$2$}; 
\fill (2,0) circle (2.3pt) node [below=2pt,blue] {\scriptsize$3$}; 
\fill (4,0) circle (2.3pt) node [below=1pt,blue] {\scriptsize$n{-}2$}; 
\fill (5,0.5) circle (2.3pt) node [below=1pt,blue] {\scriptsize$n{-}1$}; 
\fill (5,-0.5) circle (2.3pt) node [below=1pt,blue] {\scriptsize$n$}; 
\draw (5,0.5)--(4,0)--(5,-0.5);
\draw (0,0)--(1,0)--(2,0)--(2.5,0);
\draw [dashed] (2.5,0)--(3.5,0);
\draw (3.5,0)--(4,0);
\end{scope}
\end{tikzpicture}
}
\item Automorphisms: determined in~\cite{CasPar1} for $n\ge6$ and in~\cite{Sorok1} for $n=4$. The problem is open for $n=5$.
\item Endomorphisms: determined in~\cite{CasPar1} for $n\ge6$. The problem is open for $n=4,5$.
\end{itemize}

\subsubsection*{Type \texorpdfstring{$\tilde A_{n}$, $n\ge1$}{\textasciitilde An}:}
\vspace{-7ex}

\begin{itemize}
\item Coxeter graph:\qquad \raisebox{-2.5ex}{
\begin{tikzpicture}[scale=0.7, very thick]
\begin{scope}[scale=1.33] 
\fill (0,0) circle (2.3pt) node [below=2pt,blue] {\scriptsize$1$}; 
\fill (1,0) circle (2.3pt) node [below=2pt,blue] {\scriptsize$2$}; 
\fill (3,0) circle (2.3pt) node [below=1pt,blue] {\scriptsize$n{-}1$}; 
\fill (4,0) circle (2.3pt) node [below=1pt,blue] {\scriptsize$n$}; 
\fill (2,1) circle (2.3pt) node [above=2pt,blue] {\scriptsize$0$}; 
\fill [white] (2,2) circle (2.4pt);
\draw [very thick](0,0)--(1,0)--(1.5,0);
\draw [very thick,dashed] (1.5,0)--(2.5,0);
\draw [very thick](2.5,0)--(3,0)--(4,0);
\draw [very thick](0,0)--(2,1)--(4,0);
\end{scope}
\end{tikzpicture}
} \quad for $n\ge2$, \qquad  \raisebox{-2.5ex}{
\begin{tikzpicture}[scale=0.7,double distance=2.3pt,very thick]
\begin{scope}[scale=1.33] 
\fill (0,0) circle (2.3pt) node [below=2pt,blue] {\scriptsize$0$};  
\fill (1,0) circle (2.3pt) node [below=2pt,blue] {\scriptsize$1$};  
\draw (0,0)--(1,0);
\draw (0.5,0.25) node {\small $\infty$};
\end{scope}
\end{tikzpicture}
} \quad for $n=1$.

\item Automorphisms: determined in~\cite{ChaCri1} for $n\ge2$ and, independently, in~\cite{ParSor1} for $n\ge4$. $A[\tilde A_1]\simeq F_2$, the free group of rank $2$, so the case $n=1$ may be considered as known.
\item Endomorphisms: determined in~\cite{ParSor1} for $n\ge4$. The problem is open for $n=2,3$.
\end{itemize}

\subsubsection*{Type \texorpdfstring{$\tilde C_n$, $n\ge2$}{\textasciitilde Cn}:}
\begin{itemize}
\item Coxeter graph:\qquad \raisebox{-2.5ex}{
\begin{tikzpicture}[scale=0.7]
\begin{scope}[scale=1.33] 
\fill (0,0) circle (2.3pt) node [below=2pt,blue] {\scriptsize$0$};
\fill (1,0) circle (2.3pt) node [below=2pt,blue] {\scriptsize$1$};
\fill (2,0) circle (2.3pt) node [below=2pt,blue] {\scriptsize$2$};
\fill (4,0) circle (2.3pt) node [below=1pt,blue] {\scriptsize $n{-}1$};
\fill (5,0) circle (2.3pt) node [below=2pt,blue] {\scriptsize $n$};
\draw [very thick] (0,0)--(1,0)--(2,0)--(2.5,0);
\draw [very thick,dashed] (2.5,0)--(3.5,0);
\draw [very thick] (3.5,0)--(4,0)--(5,0);
\draw (4.5,0.3) node {\small $4$};
\draw (0.5,0.3) node {\small $4$};
\end{scope}
\end{tikzpicture}
}
\item Automorphisms: determined in~\cite{ChaCri1} for $n\ge2$.
\item Endomorphisms: the problem is open for all $n\ge2$.
\end{itemize}

\begin{problem}
Establish the remaining open cases in the above list. In particular, determine $\Aut(A[D_5])$.
\end{problem}

Except for type $I_2(m)$ and the original paper~\cite{DyeGro1} on $A[A_n]$, all articles cited above were relying on embedding of the respective Artin groups into mapping class groups of surfaces, and used topological methods to establish results. These methods are unavailable for the rest of spherical and affine Artin groups, so it would be desirable, if possible, to develop other methods, which do not use mapping class groups, to determine the automorphism groups of these Artin groups and describe their endomorphisms.

\begin{probstar}
Determine the automorphism group of any irreducible spherical or affine Artin group, not listed above. 
\end{probstar}

\begin{probstar}
Describe endomorphisms of any irreducible spherical or affine Artin group, not listed above. 
\end{probstar}

%%%%%%%%%%

%%%%%%%%%%

\end{document}